\documentclass[11pt]{amsart}
\usepackage{amsmath,amsthm,amsfonts,latexsym,amscd,amssymb,enumerate,color,hyperref,mathtools,bm}
\usepackage{microtype}
\usepackage[UKenglish]{babel}
\usepackage[UKenglish]{isodate}
\cleanlookdateon
\input{xypic}
\usepackage[margin=2.5cm]{geometry}
\usepackage{lineno}
\usepackage{rotating}
\usepackage{multirow}
\usepackage{multicol}
\usepackage{graphicx}
\usepackage{epstopdf}
\usepackage[usenames,dvipsnames]{xcolor}
\usepackage{inputenc}
\usepackage{tikz-cd}
\usepackage{booktabs}
\usepackage{enumitem}
\usepackage{etoolbox}
\newcounter{dummy}
\makeatletter  
\newcommand\myitem[1][]{\item[#1]\refstepcounter{dummy}\def\@currentlabel{#1}}   
\makeatother 
\usepackage{hyperref}
\usepackage{url}
\hypersetup{colorlinks,allcolors=blue}
\usepackage{comment}
\usepackage{cleveref}

\DeclareMathOperator{\cd}{CD}

\DeclareMathOperator{\Ric}{Ric}
\DeclareMathOperator{\Hess}{Hess}

\DeclareMathOperator{\vol}{vol}
\DeclareMathOperator{\diam}{diam}

\DeclareMathOperator{\supp}{supp}

\DeclareMathOperator{\RCD}{RCD}
\DeclareMathOperator{\CCD}{CD}
\DeclareMathOperator{\MCP}{MCP}

\DeclareMathOperator{\TCBA}{TCBA}
\DeclareMathOperator{\TCD}{TCD}
\DeclareMathOperator{\TMCP}{TMCP}

\newcommand{\RR}{\mathbb{R}}

\newcommand{\calF}{\mathcal{F}}

\newcommand{\calJ}{\mathcal{J}}
\newcommand{\calI}{\mathcal{I}}

\newcommand{\mfs}{\mathfrak{s}}
\newcommand{\mfc}{\mathfrak{c}}
\newcommand{\mfm}{\mathfrak{m}}

\newcommand{\mfq}{\mathfrak{q}}
\newcommand{\msd}{\mathsf{d}}
\newcommand{\eps}{\varepsilon}

\newcommand{\LLS}{(X,\msd,\ll,\leq,\tau)}
\newcommand{\LMS}{(X,\msd,\mfm,\ll,\le,\tau)}

\numberwithin{equation}{section}
\numberwithin{table}{section}

\newtheorem{theorem}{Theorem}[section]
\newtheorem{lemma}[theorem]{Lemma}
\newtheorem{proposition}[theorem]{Proposition}

\newtheorem{corollary}[theorem]{Corollary}

\theoremstyle{definition}
\newtheorem{definition}[theorem]{Definition}

\newtheorem{example}[theorem]{Example}
\newtheorem{remark}[theorem]{Remark}

\newtheorem*{ack}{Acknowledgements}
\makeatletter
\def\@setthanks{\vspace{-\baselineskip}\def\thanks##1{\@par##1}\thankses}
\makeatother

\usepackage[normalem]{ulem}

\title[Doubling for chronological diamonds in Lorentzian geometry]{Doubling for chronological diamonds in Lorentzian geometry}

\author[Che]{Mauricio Che}
\address[Che]{Faculty of Mathematics, University of Vienna, Oskar-Morgenstern-Platz 1, 1090 Wien, Austria}
\email{mauricio.adrian.che.moguel@univie.ac.at}

\author[Gieger]{Sebastian Gieger}
\address[Gieger]{Faculty of Mathematics, University of Vienna, Oskar-Morgenstern-Platz 1, 1090 Wien, Austria}
\email{sebastian.gieger@univie.ac.at}

\author[S\"amann]{Clemens S\"amann}
\address[S\"amann]{Faculty of Mathematics, University of Vienna, Oskar-Morgenstern-Platz 1, 1090 Wien, Austria}
\email{clemens.saemann@univie.ac.at}

\date{\today}

\begin{document}

\begin{abstract}
We define doubling conditions for measured Lorentzian length spaces in terms of chronological diamonds, and prove that such conditions are implied by suitable timelike curvature bounds. In particular, for the first time, we relate doubling to curvature bounds in Lorentzian geometry. As a consequence, we obtain compactness results for Lorentzian generalized cones with respect to the Lorentzian Gromov--Hausdorff convergence introduced by Mondino--S\"amann. 
\bigskip

\noindent
\emph{Keywords:} Metric geometry, Lorentzian geometry, Lorentzian length spaces, curvature bounds, doubling measures, Gromov's compactness
\medskip

\noindent
\emph{MSC2020:}
53C23, 
51K10, 
53C50, 
53B30 
\end{abstract}

\maketitle

\section{Introduction}

A metric measure space $(X,d,\mfm)$ is doubling if it is possible to uniformly bound the quotient $\mfm(B_{2r}(x))/\mfm(B_{r}(x))$ over the space. This condition is relevant in metric geometry, among other reasons, because it yields compactness theorems about classes of spaces satisfying natural geometric constraints. Most notably, the celebrated Gromov's compactness theorem, which asserts that, for any $K\in\mathbb{R}$, $N>0$, $D>0$, the class of Riemannian manifolds with Ricci curvature bounded below by $K$, dimension bounded above by $N$ and diameter bounded above by $D$, is precompact with respect to Gromov--Hausdorff convergence \cite{Gromov-poly,Gromov-1981}, is an example of a result that can be readily obtained as a consequence of doubling conditions. Moreover, this particular result has motivated the study of synthetic curvature bounds, yielding the very fruitful theories of Alexandrov spaces (see \cite{BBI,BGP}) and $\CCD$/$\RCD$ spaces (see \cite{ambrosio-gigli-savare,gigli,lott-villani,sturm1,sturm2}).

In Lorentzian signature, a synthetic theory has been proposed by Kunzinger--S\"amann in \cite{kunzinger-saemann2018} with the notion of Lorentzian (pre-)length space as the main object. In this direction, very recently, Mondino--S\"amann obtained a compactness criterion with respect to a Lorentzian version of Gromov--Hausdorff convergence \cite{mondino-saemann2025}. Crucially, this result is stated in terms of covers of Lorentzian length spaces consisting of causal diamonds, effectively playing the role of $\varepsilon$-nets in its positive-signature counterpart. 

With that motivation in mind, in this paper we aim to establish doubling conditions in the setting of Lorentzian length spaces, in terms of chronological diamonds. Our main result is the following doubling condition, which can be understood as a version of the Bishop--Gromov inequality in Lorentzian length spaces satisfying Ricci and sectional curvature bounds in the sense of \cite{cavalletti-mondino2024,kunzinger-saemann2018}. Moreover, this is new even for smooth weighted Lorentzian manifolds.

\begin{theorem}\label{thm:main1}
Let $k,K\in\mathbb R$, $N,\rho>1$, and $0<\Lambda<D_k/2$, where $D_k=\frac{\pi}{\sqrt{-k}}$ for $k<0$ and $D_k=+\infty$ otherwise. Then there exists a constant $C=C(k,K,N,\rho,\Lambda)>0$ such that the following holds: Let $X$ be a globally hyperbolic, timelike non-branching, regular measured Lorentzian length space. Assume that $X$ and its causally reversed  structure satisfy the $\TMCP^e(K,N)$-condition, and that $X$ satisfies the global $\TCBA(k)$-condition. If $I_1\subset I_2$ are chronological diamonds in $X$ such that $\diam^\tau(I_2)\leq \min(\Lambda,\rho\diam^\tau(I_1))$, then 
\[
\mfm(I_2)\leq C\, \mfm(I_1).
\]
\end{theorem}

The conditions on the timelike curvature in Theorem~\ref{thm:main1} are not as general as one would wish for. Moreover, by \cite{dajczer-nomizu80}, any Lorentzian manifold with dimension at least 3 and timelike Ricci curvature bounded in absolute value must be an Einstein manifold. Therefore, our conditions restricted to the smooth setting, e.g. $X = (M,d_g,\vol_g)$ for some Lorentzian manifold $(M^n,g)$ with $n\geq 3$, would imply that $M$ is an Einstein manifold. However, this does not rule out more general Lorentzian manifolds with timelike curvature bounded above and a timelike Bakry--Emery--Ricci tensor bounded below, which by work of Cavalletti--Mondino \cite{cavalletti-mondino2024} (see also \cite{mccann2020}), is equivalent to a $\TMCP^e$-condition for the corresponding weighted volume measure. Therefore, our work covers this more general setting (see Section \ref{ss:generalized cones bishop gromov}).

We then introduce a notion of $\eps$-separated sets, i.e., sets of pairwise disjoint chronological diamonds $\{I(p_i,q_i)\}_{i\in \Omega}$, each contained in a common set and $\tau(p_i,q_i) = \eps$ for all $i\in\Omega$ (see \Cref{def:eps-separated set}). We prove the following uniform boundedness result for such $\varepsilon$-separated sets.

\begin{theorem}\label{thm:main2}
For any $k,K\in\RR$, $N>1$, $0<\eps<\Lambda<D_k/2$, there is a constant $\widetilde{C}>0$ such that the following holds: Let $X$ 
be a globally hyperbolic, timelike non-branching, regular, measured Lorentzian length space, and $X$ (and its causally reversed structure) satisfies the $\TMCP^e(K,N)$- and global $\TCBA(k)$-conditions. Let $\tilde{I}$ be a chronological diamond such that $\diam^\tau(\tilde{I})\leq \Lambda$. Then any $\eps$-separated set in $\tilde{I}$ has cardinality at most $\widetilde{C}$.
\end{theorem}

We then apply Theorems~\ref{thm:main1} and \ref{thm:main2} to the class of Lorentzian generalized cones as described in \cite{alexander-graf-kunzinger-saemann2023}. These spaces cover a wide family of spacetimes and have the particular property that there is a distinguished time direction. Namely, these spaces are of the form $\prescript{-}{}J\times_f Y$, where $J$ is an interval, $f\colon J\to [0,\infty)$ is a continuous function, and $Y$ is a metric space, and for any $x\in Y$ the curve $\hat{\gamma}_x\colon t\mapsto (t,x)$ is a future-directed distance realiser (see \Cref{ss:generalized cones}). We can thus consider the following collection of chronological diamonds: 
\[
\hat{\calI}(X) = \{I((t,x),(t',x)): t\leq t',\ x\in Y\}.
\]
We then obtain the following uniform boundedness theorem for $\eps$-nets. Here, an $\eps$-net is a collection of causal diamonds with time separation bounded by $\eps$ that cover a given set.

\begin{theorem}\label{thm:main3}
For any $M>m>0$, $k,K\in\RR$, $N>1$, $0<\eps<\Lambda<D_k/2$, there is a constant $\hat{C}>0$
such that the following holds: Let $X=\prescript{-}{}J \times_f Y$ be a globally hyperbolic, timelike non-branching, regular, measured Lorentzian length space, $m\leq f \leq M$, and $X$ (and its causally reversed structure) satisfies the $\TMCP^e(K,N)$- and global $\TCBA(k)$-conditions. Let $\tilde I\in \hat{\calI}(X)$ be such that $\diam^\tau(I)\leq \Lambda$. Then there exists an $\eps$-net of $\tilde I$ with cardinality at most $\hat{C}$.
\end{theorem}

As one would expect from the metric case, such uniform bounds on $\eps$-nets yield compactness results with respect to the Lorentzian Gromov--Hausdorff convergence in the sense of Mondino--S\"amann \cite{mondino-saemann2025}.


\begin{theorem}\label{thm:main4}
For any $M>m>0$, $k,K\in\RR$, $N>1$, the class of covered Lorentzian pre-length spaces $(X,\mathcal{U})$ described below is precompact with respect to the Lorentzian Gromov--Hausdorff convergence as defined in \cite[Definition 3.12]{mondino-saemann2025}: 
\begin{itemize}
\item $X=I((\bar s, \bar y),(\bar t,\bar y))$
is a subset of the globally hyperbolic, timelike non-branching, regular Lorentzian length space $\prescript{-}{}J\times_f Y$, which is a generalized cone over a metric space $Y$ with $m\leq f \leq M$, $\diam^\tau(X)\leq\Lambda<D_k/2$, the global $\TCBA(k)$-condition holds, and $X$ can be equipped with a non-negative, full-support Radon measure $\mfm$ such that it (and its causally reversed structure) satisfies the $\TMCP^e(K,N)$-condition.
\item $\mathcal{U}=\{U_n\}_{n\in\mathbb N}=\{I((\bar s+1/n,\bar y),(\bar t-1/n,\bar y))\}_{n\in \mathbb{N}}$ for some fixed $\bar y\in Y$.
\end{itemize}
\end{theorem}

Finally, we obtain a few variants of \Cref{thm:main4}. Namely, in \Cref{thm:main5}  we consider Lorentzian length spaces of the form $X=\prescript{-}{}\RR \times_f Y$ satisfying analogous conditions to those of \Cref{thm:main4}. In \Cref{thm:main6} we exchange the assumption of the global $\TCBA(k)$-condition with the condition that $f\equiv 1$, i.e.\ $X$ is a Minkowski product of the form $\prescript{-}{}J\times Y$. In \Cref{thm:main7} we consider spaces of the form $X = \prescript{-}{}J\times_f Y$ where $Y$ satisfies an $\MCP$-condition, and in \Cref{cor:main8} we obtain such condition a posteriori by assuming that $X = \prescript{-}{}J\times_f Y$ satisfies a $\TCD^e_p$-condition and is equipped with a measure of the form $\mfm = f(t) dt\otimes d\mfm_Y$ where $f$ satisfies suitable conditions.

Note that other geometric compactness theorems for Lorentzian spaces already exist in the literature. As mentioned above, Mondino--S\"amann obtained a general compactness criterion for the Lorentzian Gromov--Hausdorff convergence, but also a compactness theorem for globally hyperbolic spacetimes with timelike sectional curvature bounded below as well as bounds on a compact Cauchy hypersurface for the Ricci curvature tensor, the second fundamental form and the diameter \cite[Theorem~8.9]{mondino-saemann2025}. In contrast, our Theorems~\ref{thm:main4}, \ref{thm:main5} and \ref{thm:main6} assume only curvature bounds on the ambient space and exploit the doubling conditions for chronological diamonds that we obtain in previous sections.

On the other hand, in their recent manuscript, Ketterer obtained a compactness result for smooth generalized cones with respect to a variant of the Lorentzian Gromov--Hausdorff convergence where the background metric plays a significant role 
\cite[Theorem~1.4]{ketterer2026}. This should be compared with our \Cref{thm:main7} and \Cref{cor:main8}, where we consider more general spaces, and the notion of convergence is the one introduced by Mondino--S\"amann.

Finally, compactness criteria have also been established for other notions of convergence of synthetic Lorentzian spaces, including bounded Lorentzian metric spaces \cite{minguzzi-suhr, BMS:25} and timed-metric spaces \cite{che-perales-sormani}, which builds on the null distance of Sormani--Vega \cite{SV:16} and the notion of $\tau-H$ convergence introduced by Sakovich--Sormani \cite{SS:24}. Understanding the relationship between these compactness results and the geometric compactness theorems developed here would be an interesting direction for future research.



\begin{ack}
This research was funded in whole or in part by the Austrian Science Fund (FWF) [Grant DOI’s: 10.55776/EFP6, 10.55776/STA32]. For open access purposes, the authors have applied a CC BY public copyright license to any author-accepted manuscript version arising from this submission.
\end{ack}

\section{Preliminaries}\label{s:preliminaries}

\subsection{Lorentzian pre-length spaces}\label{ss:lpls}
In this subsection we review notation, definitions and results about Lorentzian pre-length spaces that are relevant for this work. Lorentzian pre-length spaces are  analogues of metric spaces in the Lorentzian setting, while Lorentzian length spaces are analogues of metric length spaces. Introduced in \cite{kunzinger-saemann2018} (after earlier works of Busemann \cite{Bus:67} and Kronheimer--Penrose \cite{KP:67}), their precise definitions has evolved over recent years and by now there are different variants of the basic axiomatization of these spaces, cf.\ \cite{McC:24, BMcC:23, BBCGMORS:24, mondino-saemann2025} (and approaches \cite{SV:16, SS:24} based on the null distance and the bounded Lorentzian metric spaces \cite{minguzzi-suhr, BMS:25}). We opted for sticking to the original setting of \cite{kunzinger-saemann2018} as, for example, the topological and metric structure of generalized cones is always given by the cone or product structure.

Here we gather some basic definitions about Lorentzian pre-length spaces. The most fundamental structure here is that of causal spaces.
\begin{definition}\label{def:causal-space}
A \emph{causal space} is a tuple $(X,\ll,\leq)$ where $X$ is a set, and $\ll, \leq$ are binary relations, where both are transitive, $\leq$ is reflexive, and $\ll\subset \leq$.  Chronological and causal futures and pasts, as well as chronological and causal diamonds, are defined as usual:
\begin{align*}
I^{+}(p) &:= \{ q \in X : p \ll q \}, & I^{-}(q) &:= \{ p \in X : p \ll q \}, \\
J^{+}(p) &:= \{ q \in X : p \leq q \}, & J^{-}(q) &:= \{ p \in X : p \leq q \}, \\
I(p,q) &:= I^{+}(p) \cap I^{-}(q), & J(p,q) &:= J^{+}(p) \cap J^{-}(q).
\end{align*}
\end{definition}

Lorentzian pre-length spaces are obtained when we endow a causal space with a metric topology and a time separation function as follows.
\begin{definition}\label{def:lorentzian-prelength}
A \emph{Lorentzian pre-length space} $(X, d, \ll, \leq, \tau)$ is a causal space $(X, \ll, \leq)$ together with a metric $\msd$ on $X$ and a map $\tau \colon X \times X \to [0,\infty]$ satisfying:
\begin{enumerate}
    \item $\tau$ is lower semi-continuous with respect to $d$;
    \item $\tau(p,r) \geq \tau(p,q) + \tau(q,r)$ for $p \leq q \leq r$;
    \item $\tau(p,q) > 0$ if and only if $p \ll q$.
\end{enumerate}
\end{definition}

Throughout this paper we often denote a Lorentzian pre-length space $\LLS$ by $X$. We say that $X$ is:
\begin{itemize}
    \item \emph{intrinsic} if the time separation between causally related points is the supremum of lengths of causal curves connecting them;   
    \item a \emph{Lorentzian length space} if it is causally path connected \cite[Definition~3.1]{kunzinger-saemann2018}, locally causally closed \cite[Definition~3.4]{kunzinger-saemann2018}, localizable \cite[Definition~3.16]{kunzinger-saemann2018} and intrinsic;
    \item \emph{geodesic} if for all $p\le q$ there exists a \emph{distance realiser} (or \emph{geodesic}) connecting them, i.e.\ a future-directed causal curve $\gamma$ from $p$ to $q$ such that 
    \[
    L_{\tau}(\gamma) = \tau(p,q);
    \] 
    where $L_\tau$ is the $\tau$-length functional.
\end{itemize}
A crucial implication relating these concepts is that any globally hyperbolic Lorentzian length space is geodesic \cite[Theorem~3.30]{kunzinger-saemann2018} and that all timelike geodesics can be parametrized with respect to $\tau$-arclength \cite[Corollary~3.25]{kunzinger-saemann2018}, which we will assume from now on.

We say that $X$ is \textit{timelike non-branching} if whenever $\gamma_1,\gamma_2\colon [0,\eps)\to X$ are distance realisers such that $\gamma_1|_{[0,t]}=\gamma_2|_{[0,t]}$ for some $t\in (0,\eps)$, it follows that $\gamma_1=\gamma_2$.



A \emph{measured Lorentzian (pre-)length space} is a Lorentzian (pre-)length space on a Polish space $X$ endowed with a non-negative Radon measure $\mfm$ (i.e.\ a Borel-regular measure which is
finite on compact sets) such that $\supp(\mfm) = X$.

Finally, for any $S\subset X$, we define the \textit{timelike diameter} 
\[\diam^\tau(S) = \sup\{\tau(x,y):x,y\in S\}.\] In particular, if $x\leq y$ and $J(x,y)$ denotes the corresponding causal diamond, then the reverse triangle inequality for $\tau$ implies $\diam^\tau(J(x,y)) = \tau(x,y)$.

\subsection{Timelike sectional curvature bounds}

In this section we review the notion of timelike sectional curvature bounds for Lorentzian pre-length spaces introduced by Kunzinger--S\"amann \cite{kunzinger-saemann2018}, which in turn  is based on previous work by Alexander--Bishop \cite{alexander-bishop2008} and Harris \cite{harris1982} in the smooth setting. This notion is the Lorentzian analogue of sectional curvature bounds in the sense of Alexandrov (see, for example, \cite{BBI,BGP,BH}).

\begin{definition}\label{def:model-spaces}
For any integers $0\leq m\leq n$, denote by $\mathbb{R}^{n}_{m}$ the vector space $\mathbb{R}^{n}$ together with the inner product
\[
b(v,w) \;=\; -\sum_{i=1}^{m}v_i w_i + \sum_{i=m+1}^{n} v_i w_i, 
\quad \text{for } v=(v_1,\dots,v_{n}), \, w=(w_1,\dots,w_{n}).
\]
Let $k \in \mathbb{R}$. The \emph{Lorentzian $k$-planes}, or the \emph{comparison spaces} of constant curvature $k$, are defined as follows:
\begin{itemize}
    \item For $k>0$, let 
    $\mathbb{L}^2(k)$ be the universal cover of the two-dimensional de Sitter spacetime i.e. $\{ v \in \mathbb{R}^3_1 \colon b(v,v) = 1/k^2 \}$.
    
    \item For $k<0$, let 
    $\mathbb{L}^2(k)$ be the universal cover of the two-dimensional anti-de Sitter spacetime i.e. $\{ v \in \mathbb{R}^3_2 \colon b(v,v) = -1/k^2 \}$.
    
    \item For $k=0$, let $\mathbb{L}^2(0) = \mathbb{R}^2_1$, the two-dimensional Minkowski spacetime.
\end{itemize}
The \textit{timelike diameter} of $\mathbb{L}^2(k)$ is
\[
D_k = 
\begin{cases}
\pi/\sqrt{-k}, & k < 0, \\
\infty, & k \geq 0.
\end{cases}
\]
\end{definition}

\begin{definition}\label{def:geodesic-triangles}
Let $X$ be a Lorentzian pre-length space.  
\begin{enumerate}
    \item A \emph{timelike (geodesic) triangle} $\triangle p_1p_2p_3$ consists of three points 
    \[
    p_1 \ll p_2 \ll p_3 \in X, \qquad \tau(p_i,p_j) < \infty \;\; \text{for } i < j,
    \]
    together with three future-directed causal distance realising curves $\alpha_{ij}$ connecting $p_i$ to $p_j$ (for $i<j$).  
    Analogously, one defines a \emph{causal triangle}.
    
    \item An \emph{admissible causal (geodesic) triangle} $\triangle p_1p_2p_3$ consists of three points
    \[
    p_1 \ll p_2 \leq p_3 \quad \text{or} \quad p_1 \leq p_2 \ll p_3 \in X, \qquad \tau(p_i,p_j) < \infty \;\; \text{for } i < j,
    \]
    together with three (possibly constant) future-directed causal distance realising curves $\alpha_{ij}$ connecting $p_i$ to $p_j$ (for $i<j$).  

    \item A side between two vertices $p_i \ll p_j$ is called a \emph{timelike side} (even though the realising curve need not itself be timelike).  


\end{enumerate}
\end{definition}

\begin{definition}
Let $X$ be a Lorentzian pre-length space, $\mathbb{L}^2(k)$ the $k$-plane with time separation function $\bar{\tau}$, and let $\triangle p_1 p_2 p_3$ be an admissible causal triangle in $X$.  

A \emph{comparison triangle} for $\triangle p_1p_2p_3$ in the Lorentzian $k$-plane is an admissible triangle $\triangle \bar{p}_1 \bar{p}_2 \bar{p}_3$ in $\mathbb{L}^2(k)$ such that
\[
\tau(p_i, p_j) = {\bar\tau}(\bar{p}_i,\bar{p}_j) \quad \text{for } i < j.
\]

We say that the vertex $p_i$ in $X$ \emph{corresponds} to the vertex $\bar{p}_i$ in the $k$-plane. Moreover, the side $\alpha_{ij}$ connecting $p_i$ to $p_j$ (for $i<j$) in $X$ \emph{corresponds} to the side $\overline{\alpha}_{ij}$ connecting $\bar{p}_i$ to $\bar{p}_j$ in the $k$-plane. For a point $q$ lying on a timelike side $\alpha_{ij}$ of an admissible causal triangle in $X$, we define the corresponding point $\bar{q}$ on the comparison side $\overline{\alpha}_{ij}$ in the $k$-plane by requiring that it has the same time separation from the endpoints of the side. That is, we require
\[
\tau(p_i, q) = {\bar\tau}(\bar{p}_i, \bar{q}),
\]
which then automatically implies
\[
\tau(q, p_j) = {\bar\tau}(\bar{q},\bar{p}_j),
\]
since $\alpha_{ij}$ and $\overline{\alpha}_{ij}$ are distance realisers.
\end{definition}


\begin{definition}\label{def:timelike-curvature-bounds}
    Let $X$ be a globally hyperbolic Lorentzian length space and $k\in\mathbb R$. Then we say that $X$ has \emph{global timelike curvature bounded from above by $k$} (or the global $\TCBA(k)$-condition for short) if for all admissible causal triangles $\triangle xyz$ satisfying timelike size bounds for $k$ (i.e.\ all side lengths are less than $D_k$), the following holds:

    Let $p,q$ be points on the timelike sides of $\triangle xyz$ and let $\triangle\bar x\bar y\bar z$ be a comparison triangle in the comparison space $\mathbb L^2(k)$. If $\bar p,\bar q$ correspond to $p,q$ then
    \begin{equation}\label{eq:curvature bound definition}
    \tau(p,q)\geq\bar\tau(\bar p,\bar q).
    \end{equation}
\end{definition}

\begin{remark}\label{rem:curvaturebound}
\begin{enumerate}[label=(\roman*)]
    \item Note that in the usual definition of sectional curvature bounds in Lorentzian pre-length spaces we additionally assume that, locally, $\tau$ is continuous and any two timelike related points can be joined by a distance realiser. These assumptions are not necessary for our global curvature bound since global hyperbolicity already implies both conditions (see \cite[Theorem 3.28]{kunzinger-saemann2018} and \cite[Theorem 3.30]{kunzinger-saemann2018}).
    \item An immediate consequence of this definition is that whenever $\bar p\ll \bar q$ holds for two points on the comparison triangle then also $p\ll q$ in the original space. By \cite[Theorem~4.2]{beran-kunzinger-rott2024} the same implication still holds for the causal relation i.e. $\bar p\leq \bar q$ in the comparison space implies $p\leq q$ in the original one.\label{rem:curvaturebound2}
    \item There are several ways to define sectional curvature bounds for Lorentzian pre-length spaces. One commonly used version only requires \eqref{eq:curvature bound definition} to hold for timelike triangles but this is equivalent to our definition by \cite[Theorem~4.2]{beran-kunzinger-rott2024}. See also \cite{BHS:26} for recent improvements.
\end{enumerate}
\end{remark}

    
    

\subsection{Timelike measure contraction property and localization technique}
The following definition, introduced by Cavalletti and Mondino in \cite{cavalletti-mondino2024}, is the natural Lorentzian analogue of the classic measure contraction property for metric measure spaces \cite{ohta2007,sturm2}. We will not recall all the notions involved in this definition (e.g.\ the entropy functional, $\ell^p$-geodesics), therefore we refer the reader to \cite{cavalletti-mondino2020} for details.
\begin{definition}\label{def:tmcp}
A measured Lorentzian pre-length space $\LMS$ is said to satisfy the \textit{$\TMCP^e(K,N)$-condition}, for some $K \in \mathbb{R}$ and $N > 0$, if the following holds:  

For any $\mu_0 \in \mathcal{P}_c(X) \cap \mathrm{Dom}(\mathrm{Ent}(\cdot \mid \mfm))$, any $x_1 \in X$ such that 
\[
x \ll x_1 \quad \text{for $\mu_0$-a.e.\ $x \in X$,}
\]
and some (thus any; see \cite[Remark 3.8]{cavalletti-mondino2024}) $p\in (0,1)$, there exists an $\ell^p$-geodesic $(\mu_t)_{t \in [0,1]}$ from $\mu_0$ to $\mu_1 = \delta_{x_1}$ such that
\[
U_N(\mu_t \mid \mfm) \;\geq\; \sigma^{(1-t)}_{K/N}\!\big( \, \|\tau(., x_1)\|_{L^2(\mu_0)} \, \big) \, U_N(\mu_0 \mid \mfm),
\qquad \forall \, t \in [0,1).
\]
\end{definition}

The following result summarizes parts of \cite[Section 4]{cavalletti-mondino2024} that we will use in our proof of \Cref{thm:main1}.

\begin{theorem}\label{thm:needle-decomposition}
Let $\LMS$ be a globally hyperbolic, timelike non-branching, regular measured Lorentzian length space satisfying the $\TMCP^e(K, N)$-condition for some 
$K \in \mathbb{R}$, $N > 1$, and assume that the causally-reversed structure satisfies the same conditions. 

Let $x_o\in X$. 
Then 
the following disintegration formula holds:
\begin{equation}
    \mfm\llcorner_{I^+({x_o})} = 
    \int_Q \mfm_{\alpha} \, d\mfq(\alpha) 
\end{equation}
where:
\begin{itemize}
    \item $\mfq$ is a Borel probability measure over a Borel set $Q \subset I^+(x_o)$;
    \item $\{X_\alpha\}_{\alpha\in Q}$ is a partition of $I^+(x_o)$ such that $X_\alpha\cap Q = \{\alpha\}$ for all $\alpha\in Q$, and the closure $\overline{X}_\alpha$ of $X_\alpha$ is the image of a timelike geodesic $\gamma_\alpha\colon I_\alpha\to X$ starting at $x_o$, parametrized by arclength on a possibly unbounded closed interval $I_\alpha = [0,d_\alpha]$;
    \item for $\mfq$-a.e. $\alpha\in Q$, $\mfm_\alpha$ is a Radon measure on $X$ concentrated on $X_\alpha$ and absolutely continuous with respect to the Lebesgue measure restricted to $X_\alpha$;
    \item the map $\alpha \mapsto \mfm_\alpha(A)$ 
    is $\mfq$-measurable for every $\mfm$-measurable set $A$.
    \item 
    $(X_\alpha, |\cdot|, \mfm_\alpha)$ satisfies the $\mathrm{MCP}(K, N)$-condition, for $\mfq$-a.e. $\alpha\in Q$.
\end{itemize} 
\end{theorem}


\subsection{Lorentzian generalized cones}\label{ss:generalized cones}
An important class of Lorentzian pre-length spaces is given by the so-called \textit{Lorentzian generalized cones}. We briefly recall this construction here and refer the reader to \cite{alexander-graf-kunzinger-saemann2023} for details.

Let $(Y,d_Y)$ be a metric space and $I \subseteq \mathbb{R}$ an open interval. Set $X := I \times Y$ and equip $X$ with the product metric
\[
    D\bigl((t,y),(t',y')\bigr) := |t - t'| + d_Y(y,y'), 
    \quad (t,y),(t',y') \in X.
\]
Let $f \colon I \to (0,\infty)$ be continuous.  Let $\gamma \colon J \to X$ be an absolutely continuous curve with respect to $D$.  Such a curve has components $\gamma = (\alpha,\beta)$, where 
\[
\alpha \colon J \to I 
\quad \text{and} \quad 
\beta \colon J \to Y
\]
are both absolutely continuous, and the metric derivative of $\beta$, $v_\beta$, exists almost everywhere (cf.~\cite[Theorem~1.1.2]{AGS05}). We additionally require that $\alpha$ is strictly monotone.  The curve $\gamma$ is called 
\[
\begin{cases}
    \text{timelike}\\
    \text{null}\\
    \text{causal}
\end{cases}
\quad \text{ if }\quad 
-\dot\alpha^2+(f\circ \alpha)^2v_\beta^2 \quad
\begin{cases}
    < 0, \\
    = 0, \\
    \leq 0,
\end{cases}
\]
almost everywhere. It is called future/past directed causal if $\alpha$ is strictly monotonically increasing/decreasing, i.e., $\dot\alpha > 0$ or $\dot\alpha < 0$ almost everywhere.

We can define causal relations on $X$ as follows. Let $x, x' \in X$, then $x$ and $x'$ are timelike (causally) related, and we write $x\ll x'$ (resp. $x\leq x'$) if there exists a future-directed timelike (resp. causal) curve from $x$ to $x'$. We also declare $x\leq x$ for any $x\in X$.

Moreover, given a causal curve $\gamma=(\alpha,\beta)\colon [a,b]\to X$, we define its length by 
\[
L(\gamma) = \int_{a}^{b}\sqrt{\dot\alpha^2-(f\circ \alpha)^2v_\beta^2}
\]
and given $x\leq y$ in $X$, we set
\[
\tau(x,y) = \sup\{L(\gamma):\gamma\ \text{ is a causal curve from $x$ to $y$}\}.
\]
Then $(X,D,\ll,\leq,\tau)$ is a Lorentzian pre-length space, which we denote by $\prescript{-}{}J\times_f Y$.

\section{Doublings and Enlargements of Chronological Diamonds}

In this section we will prove Theorem~\ref{thm:main1}. 
While in metric spaces it is easy to describe what we mean by the ``doubling'' of a ball (namely the ball with the same center and twice the radius) defining an analogous notion in the Lorentzian setting is not straightforward. Recently, Cavalletti--Mondino proved a timelike Bishop--Gromov inequality for sublevel sets of $\tau$ \cite[Proposition~3.5]{cavalletti-mondino2024}, and building on this, McCann--S\"amann proved that the measure of such sublevel set satisfies a doubling property \cite[Lemma~5.1]{mccann-saemann2022}. However neither of these results work with chronological or causal diamonds, which are natural objects to work with in Lorentzian geometry, especially aiming towards compactness results with respect to the Lorentzian Gromov--Hausdorff convergence. 

One potential way to define the doubling of a chronological diamond $I_1=I(x,y)$ is to assume that there exists a timelike geodesic $\gamma:[0,1]\to X$ such that $\gamma(0)=x,\ \gamma(1)=y$ and $\gamma$ can be prolonged to a geodesic $\gamma:[0,2]\to X$. Then we could call $I_2:=I(\gamma(0),\gamma(2))$ the doubling of $I(\gamma(0),\gamma(1))=I_1$. While this definition is quite natural and convenient to work with, for it to work one has to assume a lot of regularity for the space. A more general definition is to call any chronological diamond $I_2=I(p,q)$ such that $I_1\subset I_2$ and $\tau(p,q)\leq 2\, \tau(x,y)$ a doubling of $I_1$. 
Then, given bounds on the curvature and the dimension of our space, it is our goal to find a constant $C>0$ such that $\mfm(I_2)\leq C\, \mfm(I_1)$ holds. 
More generally, given $\rho>1$, we want to show that there exists a constant $C>0$ depending on the curvature and the dimension of the space, as well as $\rho$, such that for any two diamonds $I_1\subset I_2$ with $\diam^\tau (I_2)\leq\rho\, \diam^\tau (I_1)$, the following holds:
\[
\mfm(I_2)\leq C\, \mfm(I_1).
\]
Let us make all of this precise.

\begin{definition}\label{def:enlargement}
    Let $\rho>1$ and let $I_1$ be a chronological diamond. We call a chronological diamond $I_2\supset I_1$ a \emph{$\rho$-enlargement} of $I_1$ if 
    $$\diam^\tau(I_2)\leq\rho\, \diam^\tau(I_1).$$
    If furthermore we have $I_1=I(x_1,y_1),\ I_2= I(x_2,y_2)$ with $x_1,x_2,y_1$ and $y_2$ all lying on one timelike geodesic and either $x_1=x_2$, or $y_1=y_2$, we call $I_2$ an \emph{aligned $\rho$-enlargement} of $I_1$. 
    
    Let $\Lambda\in(0,\infty]$. We say that a measured Lorentzian pre-length space has the \emph{$(\rho,\Lambda)$-enlargement property} with constant $C$ if for all $I_1\subset I_2$ such that $I_2$ is a $\rho$-enlargement of $I_1$ and $\diam^\tau(I_2)\leq\Lambda$ we have
    $$\mfm(I_2)\leq C\, \mfm(I_1).$$
    If this only holds for aligned $\rho$-enlargements, we say that the space has the \emph{aligned $(\rho,\Lambda)$-enlargement property}. Finally if $\rho=2$ we call $I_2$ an (aligned) \emph{doubling} of $I_1$ and if in addition $\Lambda=\infty$ we say that the space has the (aligned) \emph{doubling property} with doubling constant $C$.
\end{definition}

\begin{remark}
    Note that the notions of $(\rho,\Lambda)$-enlargements and aligned $(\rho,\Lambda)$-enlargements precisely correspond to the two different notions of doublings discussed in the paragraph before the definition. Also note that of course any aligned $\rho$-enlargement is in particular a $\rho$-enlargement. For this reason the $(\rho,\Lambda)$-enlargement property is stronger than the aligned $(\rho,\Lambda)$-enlargement property.
\end{remark}

The following proposition is an analogue to the known fact in the positive-signature setting that doubling Radon measures have full support (see, for example, \cite[Proposition~18.4]{villani}).

\begin{proposition}\label{pop:doubling implies full support}
Let $\LLS$ be a path-connected, strongly causal, regularly localizable, Lorentzian pre-length space, $\mfm$ a non-negative Radon measure on $X$ that is not identically zero, and such that the $(\rho,\Lambda)$-enlargement property holds for some $\rho>1$ and $\Lambda > 0$. 
Then $\supp(\mfm) = X$.
\end{proposition}
\begin{proof}
Since $\mfm$ is not identically zero, $K:=\supp(\mfm)\neq \varnothing$. First, we claim that if $I(x,y)$ is a chronological diamond such that $\tau(x,y)\leq \Lambda$ and $I(x,y)\cap K\neq \varnothing$ then $I(x,y)\subset K$:

Indeed, let $p\in I(x,y)$. Since $I(x,y)\cap K\neq \varnothing$ and $I(x,y)$ is open, $\mfm(I(x,y)) > 0$. Now let $\{I(u_j,v_j)\}_{j=1}^{N}$ be such that $p\in \bigcap_{j=1}^{N} I(u_j,v_j)=:U$. By strong causality, it is sufficient to show that $\mfm(U)>0$. 

Since $X$ is regularly localizable, distance realisers in $X$ have a causal character \cite[Theorem~3.18]{kunzinger-saemann2018}. In particular, by moving along distance realisers $[x,p]$ and $[p,y]$, which are continuous and timelike, we can find $x'\in [x,p]$ and $y'\in [p,y]$ such that $p\in I(x',y') \subset I(x,y)\cap U$. By taking $x = x'_n \ll \dots \ll x'_0 = x'$ in $[x,p]$ and $y'= y'_0\ll \dots \ll y'_n = y$ in $[p,y]$ such that
\[
\tau(x'_{i+1},y'_{i+1}) \leq \rho\, \tau(x'_i,y'_i),\qquad i=0,\dots, n-1,
\]
the $(\rho,\Lambda)$-enlargement property implies that for some $C>0$,
\[
\mfm(I(x'_{i+1},y'_{i+1}))\leq C\, \mfm(I(x'_i,y'_i)),\qquad i=0,\dots, n-1.
\]
It follows that 
\[
0< \mfm(I(x,y)) \leq C^n\, \mfm(I(x',y')) \leq C^n\, \mfm(U). 
\]
In particular, $\mfm(U) > 0$ and the claim follows. 

For general $p\in X$, let $\gamma\colon [a,b]\to X$ be a continuous path joining $p$ with any point in $K$. Then, by regular localizability and compactness of $\gamma([a,b])$, we can find $\{I(x_i,y_i)\}_{i=1}^{k}$ covering $\gamma([a,b])$ and such that $\tau(x_i,y_i)\leq \Lambda$. By connectedness of $\gamma([a,b])$ and the previous argument, it follows that $\gamma([a,b])\subset K$, and in particular $p\in K$. In other words, $X =\supp(\mfm)$.
\end{proof}

In analogy with the Riemannian and metric settings, one might expect a space to satisfy a doubling property whenever its Ricci curvature is bounded from below in timelike directions, or more generally whenever it satisfies either the $\TCD^e_p(K,N)$-condition (see \cite[Definition~3.2]{cavalletti-mondino2024}) or the $\TMCP^e(K,N)$-condition (Definition~\ref{def:tmcp}) for some constants $K,N$. Recall that if a metric measure space $X$ satisfies the $\cd(K,N)$-condition, then the doubling property holds on every bounded subset of the support of $\mfm$ \cite[Corollary 2.4]{sturm2}. We now show, however, that in the Lorentzian setting neither the $\TCD^e_p$- nor the $\TMCP^e$-condition implies a doubling property in the sense of Definition~\ref{def:enlargement}, at this level of generality.

\begin{example}
    Consider $2$-dimensional anti-de Sitter space of radius $r$. Recall that this space can be viewed as a surface in $\mathbb R^3_2$ given by the following parametrization:
    \begin{equation}\label{eq:AdS}
    \begin{split}
        &\phi:(-\pi/2,3\pi/2)\times\mathbb R\to \mathbb{R}^3_2\\
        &\phi(t,x)=(\sqrt{r^2+x^2}\sin(t),\sqrt{r^2+x^2}\cos(t),x).
    \end{split}
    \end{equation}
    This space has bounded finite diameter (namely $r\pi$) and its Ricci curvature tensor is given by $\Ric (X,Y)=-\frac{1}{r^2}\langle X,Y\rangle$. In particular the space above satisfies $\TMCP^e(\frac{1}{r^2},2)$. For simplicity, from now on consider the case $r=1$. Take the timelike geodesic $\gamma:[0,\pi]\to X$, $\gamma(t)=(\sin(t),\cos(t),0)$ and define 
    \begin{align*}
        &x_0:=\gamma(0)=(0,1,0)\\
        &y_1:=\gamma(\pi/2)=(1,0,0)\\
        &y_2:=\gamma(\pi)=(0,-1,0)
    \end{align*}
    as well as $I_1:=I(x_0,y_1),\ I_2:=I(x_0,y_2)$. With this choice $I_2$ is clearly an aligned doubling of $I_1$. A quick calculation yields that $\mfm(I_2)=\infty$, while $\mfm(I_1)<\infty$, so clearly the (aligned) doubling property fails for any $C>0$. The space we are looking at is not globally hyperbolic so perhaps this counterexample can be ignored. However let us now take the (globally hyperbolic) subspace $X:=I(x_0,y_2)$ then there still does not exist a $C>0$ such that $X$ has the doubling property with doubling constant $C$. Indeed take
    \begin{align*}
        &x_0^i:=(\sin(1/i),\cos(1/i),0),\\
        &y_2^i:=(\sin(\pi-1/i),\cos(\pi-1/i),0),
    \end{align*}
    (i.e. we take sequences of points converging to $x_0$, $y_2$ respectively) and define $I_1^i:=I(x_0^i,y_1),\ I_2^i:=I(x_0^i,y_2^i)$. Then we will have $\mfm(I_1^i)\to\mfm(I_1)<\infty$, while $\mfm(I_2^i)\to\infty$, i.e., we cannot find a uniform doubling constant $C$ for all $i$.
\end{example}

The existence of the $\Lambda$ in our definition of the $(\rho,\Lambda)$-enlargement property already suggests that what went wrong in the above example was the lack of a restriction on the size of the diamonds. To fix this we introduce the following notion:

\begin{definition}
    Given $\rho>1,\ c\in(0,1),\ \Lambda\in(0,\infty]$, we say that a pair of aligned chronological diamonds $I_1=I(x,y_1),\ I_2=I(x,y_2)$ satisfies the \textit{$(\rho,c,\Lambda)$-condition} if $\diam^\tau (I_2)\leq\Lambda$, $I_2$ is an aligned $\rho$-enlargement of $I_1$ and for all constant-speed timelike geodesics $\gamma:[0,1]\to I_2$ such that $\gamma(0)=x$, the following implication holds:
    \[
    \gamma(1)\in I_2\ \Rightarrow \ \gamma(c)\in I_1.
    \]
    We say that a Lorentzian pre-length space satisfies the \emph{$(\rho,c,\Lambda)$-condition} if all pairs of chronological diamonds $I_1,I_2$, such that $\diam^\tau (I_2)\leq\Lambda$ and $I_2$ is an aligned $\rho$-enlargement of $I_1$, satisfy the $(\rho,c,\Lambda)$-condition. 
\end{definition}

We can now prove that by imposing this additional restriction on a space, it will satisfy the $(\rho,\Lambda)$-enlargement property.

\begin{proposition}\label{prop:1}
    Let $\mathcal X$ be the class of globally hyperbolic, timelike non-branching, regular measured Lorentzian length space satisfying the $\TMCP^e(K,N)$-condition, such that the causally reversed structure satisfies the same condition and all $X\in\mathcal X$ satisfy the $(\rho,c,\Lambda)$-condition for some $\rho>1,\ c\in(0,1),\ \Lambda\in(0,\infty]$ (where we only allow $\Lambda=\infty$ if $K\geq0$).
    
    Then there exists a constant $C=C(K,N,c,\Lambda)>0$ such that all $X\in\mathcal X$ have the aligned $(\rho,\Lambda)$-enlargement property with constant $C$. 
\end{proposition}

The proof will heavily rely on \Cref{thm:needle-decomposition}. Recall that the $\MCP(K,N)$-condition on a metric measure space implies the following volume growth estimate for the measure of balls $B_r(p)$:
\begin{equation}\label{eq:bishopgromov}
    \frac{\mfm(B_R(p))}{\mfm(B_r(p))}\leq\frac{\int_0^R\mfs_{K/(N-1)}(t)^{N-1}dt}{\int_0^r\mfs_{K/(N-1)}(t)^{N-1}dt}\hspace{1cm}(0<r\leq R)
\end{equation}
where
\begin{equation*}
    \mfs_\kappa(t):=\begin{cases}
        \frac{1}{\sqrt\kappa}\sin(\sqrt{\kappa}t)&\kappa>0,\\
        t&\kappa=0,\\
        \frac{1}{\sqrt{-\kappa}}\sinh(\sqrt{-\kappa}t)&\kappa<0.
    \end{cases}
\end{equation*}
We also define 
\[
    \mfc_\kappa(t) := \mfs_\kappa'(t) =
    \begin{cases}
        \cos(\sqrt{\kappa}t)&\kappa>0,\\
        1&\kappa=0,\\
        \cosh(\sqrt{-\kappa}t)&\kappa<0.
    \end{cases}
\]

Before proving \Cref{prop:1} we first need to show the following auxiliary result. 

\begin{lemma}\label{lem:disintegration}
Let $X$ be a globally hyperbolic, timelike non-branching, regular 
measured Lorentzian length space such that $X$ and its causally reversed structure satisfy the $\TMCP^e(K,N)$-condition. Let $x\ll y_1\ll y_2$ be points in $X$ and define $J_1:=J(x,y_1),\ J_2:=J(x,y_2)$. Let $\mfq$, $Q$, and for $\mfq$-a.e. $\alpha\in Q$, $X_\alpha$, $\gamma_\alpha$, and $I_\alpha = [0,d_\alpha]$ be given by \Cref{thm:needle-decomposition} applied to $\mfm\llcorner_{I^+(x)}$. Then for $\mfq$-a.e. $\alpha\in Q$ there exist 
$0\leq t_{\alpha,1}\leq t_{\alpha,2}\leq d_\alpha$ such that $\gamma_\alpha([0,t_{\alpha,i}]) = \overline{X}_\alpha\cap J_i$, for $i=1,2$ (see \Cref{fig:t-alpha}).
\end{lemma}
\begin{figure}[h]

\def \globalscale {1.000000}
\begin{tikzpicture}[y=1.5cm, x=1.5cm, yscale=\globalscale,xscale=\globalscale, every node/.append style={scale=\globalscale}, inner sep=0pt, outer sep=0pt]
  \path[draw=black,line width=0.05cm] (4.5, 28.3) -- (1.8, 25.6) -- (4.2, 23.2) 
  -- (6.9, 25.7) -- cycle;
  \path[draw=black,line width=0.05cm] (4.4, 26.3) -- (2.8, 24.6) -- (4.2, 23.2) 
  -- (5.8, 24.68) -- cycle;
  \draw [-stealth,line width=0.05cm] (4.2, 23.2) -- (6.2, 27.9);

  \draw [fill=black] (4.2,23.2) circle (0.08cm);
  \draw [fill=black] (4.4,26.3) circle (0.08cm);
  \draw [fill=black] (4.5,28.3) circle (0.08cm);
  \draw [fill=black] (5.78,26.91) circle (0.08cm);
  \draw [fill=black] (5.15,25.43) circle (0.08cm);
  
  \node[text=black,anchor=south west,line width=0.0cm] (text4) at (6.2, 
  28){$\gamma_\alpha$};
  \node[text=black,anchor=south west,line width=0.0cm] (text5) at (5.95, 
  26.9){$\gamma_{\alpha}(t_{\alpha,2})$};
  \node[text=black,anchor=south west,line width=0.0cm] (text6) at (5.3, 
  25.4){$\gamma_{\alpha}(t_{\alpha,1})$};
  \node[text=black,anchor=south west,line width=0.0cm] (text7) at (4.1, 
  22.95){$x$};
  \node[text=black,anchor=south west,line width=0.0cm] (text9) at (4.3, 
  26.4){$y_1$};
  \node[text=black,anchor=south west,line width=0.0cm] (text10) at (4.4, 
  28.4){$y_2$};
\end{tikzpicture}
\caption{Setup of \Cref{lem:disintegration}}
\label{fig:t-alpha}
\end{figure}

\begin{proof}
By \Cref{thm:needle-decomposition}, for $\mfq$-a.e. $\alpha\in Q$, the timelike geodesic $\gamma_\alpha\colon [0,d_\alpha]\to X$ is parametrized by arclength and satisfies $\gamma_\alpha([0,d_\alpha]) = \overline{X}_\alpha$. If $\overline{X}_\alpha \subset J_1$ then we simply take $t_{\alpha,1}=t_{\alpha,2}=d_\alpha$. 

If $\overline{X}_\alpha \subset J_2$ but $\overline{X}_\alpha \not\subset J_1$, we let $t_{\alpha,2}=d_\alpha$ and define $t_{\alpha,1}$ as follows. By global hyperbolicity, we know that $\tau$ is continuous \cite[Theorem 3.28]{kunzinger-saemann2018}. Since $t\mapsto \tau(x,\gamma_\alpha(t))$ is increasing and continuous, it follows that there exist a minimal $t_{\alpha,1}\in [0,d_\alpha]$ such that $\tau(\gamma_\alpha(t_{\alpha,1}),y_1) = 0$ and since causal diamonds are closed it is easy to see that $\gamma_\alpha(t_{\alpha,1})\leq y_1$.
An analogous argument yields a minimal $t_{\alpha,2}\in[0,d_\alpha]$ such that $\tau(\gamma_\alpha(t_{\alpha,2}),y_2) = 0$ and $\gamma_\alpha(t_{\alpha,2})\leq y_2$, in the case $\overline{X}_{\alpha}\not\subset J_2$. Furthermore, in this case, if $t_{\alpha,2} < t_{\alpha,1}$ then, the definition of $t_{\alpha,1}$ and the fact that $\gamma_\alpha$ is timelike yield $\tau(\gamma_\alpha(t_{\alpha,2}),y_2) > 0$, which contradicts the definition of $t_{\alpha,2}$. Thus, $t_{\alpha,1}\leq t_{\alpha,2}$.

Finally, 
if $\overline{X}_\alpha \not\subset J_1$ then for any $t>t_{\alpha,1}$ the point $\gamma_\alpha(t)$ is no longer in $J_1$; otherwise, the reverse triangle inequality and the fact that $\tau(\gamma_{\alpha}(t_{\alpha,1}),x_1)=0$ would imply $\tau(\gamma_{\alpha}(t_{\alpha,1}),\gamma_{\alpha}(t))=0$, which contradicts that $\gamma_\alpha$ is timelike. Hence $\gamma_\alpha([0,t_{\alpha,1}])=\overline{X}_\alpha\cap J_1$. An analogous argument yields $\gamma_\alpha([0,t_{\alpha,2}])=\overline{X}_\alpha\cap J_2$ when $\overline{X}_\alpha\not \subset J_2$. 
\end{proof}

\begin{proof}[Proof of \Cref{prop:1}]
    Let $\rho>1,\ c\in(0,1),\ \Lambda\in(0,\infty]$ be given, pick some $X\in\mathcal X$ and choose $I_1\subset I_2\subset X$, $I_1 = I(x,y_1)$, $I_2=I(x,y_2)$, such that $I_2$ is an aligned $\rho$-enlargement of $I_1$ and $\diam^\tau (I_2)\leq\Lambda$. 
    
    Let $\mfq$, $Q$, $X_\alpha$ and $d_\alpha$ be given by \Cref{thm:needle-decomposition} applied to $\mfm\llcorner_{I^+(x)}$. For $\mfq$-a.e. $\alpha \in Q$, let $\gamma_\alpha:[0,d_\alpha]\to X$ and $0\leq t_{\alpha,1},\leq t_{\alpha,2}\leq d_\alpha$ be the parameters such that $\gamma_\alpha([0,t_{\alpha,i}])=\overline{X}_\alpha\cap J_i$, $i=1,2$, as in \Cref{lem:disintegration}. Let $t_n\to t_{\alpha,2}$ with $0<t_n<t_{\alpha,2}$. Since $\gamma_{\alpha}$ is timelike, $x\ll \gamma_{\alpha}(t_n)\ll \gamma_{\alpha}(t_{\alpha,2})$ and by the push-up property, $\gamma_\alpha(t_n) \in I_2$. The $(\rho,c,\Lambda)$-condition gives that $\gamma_\alpha(ct_n)\in I_1$. Global hyperbolicity and continuity of $\gamma_\alpha$ yield $\gamma_\alpha(c t_{\alpha,2})\in J_1$. Therefore, in particular, $c t_{\alpha,2}\leq t_{\alpha,1}$.

    Furthermore, since $\mfm_\alpha$ is concentrated on $X_\alpha$ and $\overline{X}_\alpha\cap I_i = \gamma_\alpha((0,t_{\alpha,i}))$ by construction of $t_{\alpha,i}$, we have
    \[
    \mfm_\alpha(X_\alpha\cap I_i) = \mfm_\alpha(\overline{X}_\alpha\cap I_i) = \mfm_{\alpha}(\gamma_\alpha((0,t_{\alpha,i}))) = \mfm_{\alpha}(\gamma_\alpha([0,t_{\alpha,i}))).
    \]
    This yields the following chain of inequalities:
    \begin{align}\label{equation1}
        \frac{\mfm_\alpha(X_\alpha\cap I_2)}{\mfm_\alpha(X_\alpha\cap I_1)}\le\frac{\mfm_\alpha(\gamma_\alpha([0,t_{\alpha,2})))}{\mfm_\alpha(\gamma_\alpha([0,ct_{\alpha,2})))} \le \frac{\int_{0}^{t_{\alpha,2}}\mfs_{K/(N-1)}(t)^{N-1} dt}{\int_{0}^{ct_{\alpha,2}}\mfs_{K/(N-1)}(t)^{N-1} dt}=:F_{K,N,c}(\alpha),
    \end{align}
    where the final inequality follows from \eqref{eq:bishopgromov}, which in turn is implied by the $\MCP(K,N)$-condition on $X_\alpha$. If $K\geq0$ it is easy to see that there exists a constant $C=C(K,N,c)$ such that $F_{K,N,c}(\alpha)\leq C$ for $\mfq$-almost all $\alpha\in Q$. If $K<0$ then $t_{\alpha,2}\leq\Lambda<\infty$ and hence there again exists $C=C(K,N,c,\Lambda)$ such that $F_{K,N,c}(\alpha)\leq C$ for $\mfq$-almost all $\alpha\in Q$. By integrating \eqref{equation1} over $\alpha\in Q$, we obtain
    \begin{equation*}
    \mfm(I_2) = \int_{Q} \mfm_\alpha(X_\alpha\cap I_2) d\mfq(\alpha)  \leq \int_Q F_{K,N,c}(\alpha)\mfm_\alpha(X_\alpha\cap I_1) d\mfq(\alpha)\leq C\int_Q \mfm_\alpha(X_\alpha\cap I_1) d\mfq(\alpha)=C\, \mfm(I_1).
    \end{equation*}
    This implies that $X$ has the aligned $(\rho,\Lambda)$-enlargement property with constant $C$. 
\end{proof}

While \Cref{prop:1} is a good first step towards proving certain spaces have the $(\rho,\Lambda)$-enlargement property, the assumption on the space to satisfy the $(\rho,c,\Lambda)$-condition is quite unnatural. We will try to replace it with a timelike sectional curvature bound by $k\in\mathbb R$ in the sense of \Cref{def:timelike-curvature-bounds}.

\begin{proposition}\label{prop:2}
    Let $\mathcal X$ be the class of globally hyperbolic Lorentzian length spaces, which satisfy the global $\TCBA(k)$-condition for some $k\in\mathbb R$. Fix any $\rho>1$ and if $k<0$ let $0<\Lambda\leq\frac{\pi}{2\sqrt{-k}}$, otherwise set $\Lambda=\infty$. 
    
    Then there exists a constant $c=c(k,\rho,\Lambda)$ such that all $X\in\mathcal X$ satisfy the $(\rho,c,\Lambda)$-condition.
\end{proposition}

\begin{proof}
     Let $X\in\mathcal X$ and let $I_1:=I(x,y_1)\subset I(x,y_2)=:I_2$ be chronological diamonds in $X$ such that $I_2$ is an aligned $\rho$-enlargement of $I_1$ and $\diam^\tau(I_2)\leq \Lambda$. Recall that this means that there exists a timelike geodesic $\gamma:[0,1]\to X$ such that $\gamma(0)=x,\ \gamma(1)=y_2$ and there exists a parameter $1\geq T_1\geq\rho^{-1}$ such that $\gamma(T_1)=y_1$.
     
     Let $\xi:[0,1]\to X$ be a timelike geodesic with $\xi(0)=x$ and $\xi(1)\in I_2$. We will consider the timelike triangle $\triangle x\xi(1)y_2$ where the side from $x$ to $y_2$ is realised by $\gamma$ and the side from $x$ to $\xi(1)$ is realised by $\xi$. The choice of geodesic going from $\xi(1)$ to $y_2$ does not matter. Choose a comparison triangle $\triangle \bar x\bar\xi(1)\bar y_2$ in the comparison space $\mathbb L^2(k)$. We claim that there exists a $c=c(k,\rho,\Lambda)$ such that $\mathbb L^2(k)$ satisfies the $(\rho,c,\Lambda)$-condition. If this is the case then $\bar\xi(c)\in I(\bar x,\bar y_1)=I(\bar x,\bar\gamma(T_1))$ since $T_1\geq \rho^{-1}$ and in particular $\bar\xi(c)\ll \bar y_1$. By \Cref{rem:curvaturebound}\ref{rem:curvaturebound2} the curvature bound now implies that $\xi(c)\ll y_1$ and in particular $\xi(c)\in I_1$, i.e.\ $X$ satisfies the $(\rho,c,\Lambda)$-condition as well. Finally, the claim holds by \Cref{lem:rhoclambda} below.
\end{proof}

\begin{lemma}\label{lem:rhoclambda}
    Let $k\in\RR, \rho>1$ and if $k<0$ let $0<\Lambda\leq\frac{\pi}{2\sqrt{-k}}$, otherwise set $\Lambda=\infty$. Then there exists a constant $c=c(k,\rho,\Lambda)$ such that $\mathbb L^2(k)$ satisfies the $(\rho,c,\Lambda)$-condition.
\end{lemma}

\begin{proof}
     We use the same setup as in the proof of \Cref{prop:2} above. Take $I_1:=I(x,y_1)\subset I(x,y_2)=:I_2$ in $\mathbb L^2(k)$ such that $I_2$ is an aligned $\rho$-enlargement of $I_1$ and $\diam^\tau(I_2)\leq \Lambda$. Let $\gamma:[0,1]\to X$ be the timelike geodesic such that $\gamma(0)=x,\ \gamma(1)=y_2$ and let $T_1$ be the parameter such that $\gamma(T_1)=y_1$. We need to show that for a timelike geodesic $\xi:[0,1]\to X$ with $\xi(0)=x$ and $\xi(1)\in I_2$, there exists a $c=c(k,\rho,\Lambda)$ such that $\xi(c)\in I_1$.
     
     First assume that $k=0$, i.e. we are considering flat, 2-dimensional Minkowski space. In this case, $c=\rho^{-1}$ works. Indeed, by the intercept theorem we know that $\xi(\rho^{-1})\ll\gamma(\rho^{-1})$. Since $\rho^{-1}\leq T_1$, we get $\xi(\rho^{-1})\ll\gamma(T_1)= y_1$, i.e., it satisfies the $(\rho,\rho^{-1},\infty)$-condition. If $k>0$ then, by \cite[Lemma 3.27]{beran-saemann2023}, $\mathbb{L}^2(k)$ satisfies the $\TCBA(0)$-condition, so the case $k=0$ combined with \Cref{prop:2} yields the result. 

    It remains to consider the case $k<0$. Here we have to make the additional assumption that $I_2$ satisfies $\diam^\tau(I_2)\leq\Lambda\leq\frac{\pi}{2\sqrt{-k}}$. Recall that $\mathbb L^2(k)$ can be seen as the manifold parametrized by \eqref{eq:AdS}, where the radius $r$ is equal to $1/\sqrt{-k}$. Without loss of generality we can assume that the curve $\gamma$ is given by 
    $$\gamma:[0,1]\to \mathbb L^2(k),\ \gamma(t)=\left(r \sin\left(\frac{\pi}{2}t\right),r \cos\left(\frac{\pi}{2}t\right),0\right),$$
    so in particular we assume that $\diam^\tau(I_2)=\Lambda=\frac{\pi}{2\sqrt{-k}}$. We will see later how the case $\diam^\tau(I_2)<\frac{\pi}{2\sqrt{-k}}$ follows from this. We will also assume that $T_1=\rho^{-1}$ with the case $T_1>\rho^{-1}$ trivially following from the former. 
    Denote by $t_1$ the parameter at which $\xi$ leaves $I_1$, then our goal is to find a constant $c\in(0,1)$, independent of the choice of $\xi$, such that 
    \begin{equation*}
        \frac{L(\xi\big|_{[0,t_1]})}{L(\xi)}> c.
    \end{equation*}
    In order to do this we need to make precise computations. Recall that the geodesic $\xi$ is given by affinely parameterizing the intersection of $\mathbb L^2(k)\subset\mathbb R^3_2$ with the linear subspaces spanned by the vectors $x=(0,r,0)$ and $\xi'(0)$. These initial velocities are multiples of the vectors $(1,0,a)$ where $a\in(-1,1)$. The case $a=0$ corresponds to the geodesic $\gamma$, while $|a|=1$ would give a null curve. Setting the parametrization of anti-de Sitter equal to this linear space and solving for $t$ gives the curve 
    \begin{equation}\label{eq:AdSgeod}
    \tilde\xi_a(t)=\frac{r}{\sqrt{1 - a^2 \sin(t)^2}}(\cos(t),\sin(t),a \sin(t)).
    \end{equation}
    Note that this is only a pregeodesic, i.e., the reparametrization of a geodesic, which does not necessarily have constant speed. Indeed the the value of $\vert\vert\tilde\xi_a'(t)\vert\vert$ is not constant but can be calculated to be
    \begin{equation}\label{eq:AdSnorm}
    \vert\vert\tilde\xi_a'(t)\vert\vert=\frac{r\sqrt{1-a^2}}{1-a^2\sin^2(t)}.
    \end{equation}
    Calling the times at which $\tilde\xi_a$ leaves $I_1$ (respectively $I_2$) $t_{a,1}$ (respectively $t_{a,2}$) our goal is to find a constant $c$ such that for all $a\in(-1,1)$ we have
    $$\frac{L(\tilde\xi_a\big|_{[0,t_{a,1}]})}{L(\tilde\xi_a\big|_{[0,t_{a,2}]})}>c.$$
    As a first step towards this we compute the values for $t_{a,i}$ by finding the intersection points of $\tilde\xi_a$ and the past directed null geodesics starting at $y_i$. It is easy to see that such a past directed null geodesic starting at $x$ is given by 
    $$t\mapsto (-t,r,t).$$
    To find such a geodesic starting at $y_i$ we just need to rotate this line by the appropriate angle $\theta_i$, where $\theta_1=\pi/4$ and $\theta_2=\pi/2$. This yields 
    $$\begin{pmatrix}\cos(\theta_i) & \sin(\theta_i) & 0 \\ -\sin(\theta_i) & \cos(\theta_i) & 0 \\ 0 & 0 & 1 \end{pmatrix}
    \cdot
    \begin{pmatrix}-t\\r\\t\end{pmatrix}
    =t\ 
    \begin{pmatrix}-\cos(\theta_i)\\\sin(\theta_i)\\1\end{pmatrix}
    +
    \begin{pmatrix}r\ \sin(\theta_i)\\ r\ \cos(\theta_i)\\0\end{pmatrix}
    =:\eta_i(t),$$
    with $i=1,2$. 
    Now $t_{a,i}$ is the parameter at which $\tilde\xi_a$ intersects $\eta_i$. For $i=2$ this can easily be computed to be
    \begin{align*}
        t_{a,2}=\arccos\left(\frac{|a|}{\sqrt{1 + a^2}}\right).
    \end{align*}
    While a precise value for $t_{a,1}$ is not as straightforward to compute it is still easy to see that on $(0,t_{a,1})$ we can estimate 
    $$\vert\vert\tilde\xi_a'(t)\vert\vert>r\sqrt{1-a^2}=:c_{a,1}.$$
    Meanwhile on $(0,t_{a,2} )$ we can estimate 
    $$\vert\vert\tilde\xi_a'(t)\vert\vert<\frac{r\sqrt{1-a^2}}{1-a^2\sin^2(t_{a,2})}=:c_{a,2}.$$
    Now we can make the following computation:
    $$\frac{L(\tilde\xi_a\big|_{[0,t_{a,1}]})}{L(\tilde\xi_a\big|_{[0,t_{a,2}]})}=\frac{\int_0^{t_{a,1}}\vert\vert\tilde\xi'_a(t)\vert\vert\ dt}{\int_0^{t_{a,2}}\vert\vert\tilde\xi'_a(t)\vert\vert\ dt}>\frac{\int_0^{t_{a,1}}c_{a,1}\ dt}{\int_0^{t_{a,2}}c_{a,2}\ dt}=\frac{c_{a,1}}{c_{a,2}}\cdot\frac{t_{a,1}}{t_{a,2}}.$$
    First we want to show that $a\mapsto\frac{t_{a,1}}{t_{a,2}}$ is bounded away from $0$. Let $a_0<1$, then this is clearly the case on the compact interval $[-a_0,a_0]$ hence it is only the behavior as $|a|\to1$ which might cause problems. Note that $t_{a,2}\to\arccos(1/\sqrt{2})=\pi/4$ as $|a|\to 1$ so in fact we only need to show that $t_{a,1}$ is bounded away from $0$ as $|a|\to1$. To this end observe that for $\varepsilon>0$ small enough, on the interval $[0,\pi/2-\varepsilon]$ the curve $\tilde\xi_a$ converges uniformly to the curve
    $t\mapsto(r\, \tan(t),r,r\, \tan(t))$
    as $a\nearrow 1$. This curve intersects $\eta_1$ at parameter $t_{1,1}:=-\arctan(\cot(T_1) - \csc(T_1))>0$ so, since the time of intersection between two curves continuously depending on a parameter is clearly a continuous function, we get $t_{a,1}\to t_{1,1}>0$ as $a\nearrow1$ and in particular $t_{a,1}$ is bounded away from $0$. The case $a\searrow-1$ works analogously. 
    
    Similarly we claim that $a\mapsto\frac{c_{a,1}}{c_{a,2}}$ is bounded away from $0$. Again it suffices to check the behavior as $|a|\to1$. The denominator in the definition of $c_{a,2}$ behaves as follows as $|a|\to1$:
    $$1-a^2\sin^2(t_{a,2})\to1-\sin^2(\pi/4)=1/2.$$
    Consequently we have 
    $$\lim_{|a|\to1}\frac{c_{a,1}}{c_{a,2}}= \lim_{|a|\to1}\frac{r\sqrt{1-a^2}}{2r\sqrt{1-a^2}}=\frac{1}{2}.$$
    In particular this means that there is a constant $c$ independent of $a$ (and also independent of the radius $r$, which will be important in the next paragraph) such that
    \begin{equation}\label{ineqc}
    \frac{L(\tilde\xi_a\big|_{[0,t_{a,1}]})}{L(\tilde\xi_a\big|_{[0,t_{a,2}]})}>\frac{c_{a,1}}{c_{a,2}}\cdot\frac{t_{a,1}}{t_{a,2}}\geq c,
    \end{equation}
    which is what we wanted to show. 

    It remains to prove the statement for $\diam^\tau(I_2)<\frac{\pi}{2\sqrt{-k}}$. 
    There exists a $\bar k<k$ such that $\diam^\tau(I_2)=\frac{\pi}{2\sqrt{-\bar k}}$, so in particular we can find timelike geodesics $\bar\gamma,\bar\xi:[0,1]\to\mathbb L^2(\bar k)$, starting at the same point, such that $\triangle\bar\gamma(0)\bar\xi(1)\bar\gamma(1)$ is a comparison triangle for $\triangle\gamma(0)\xi(1)\gamma(1)$.
    By \eqref{ineqc} we know that $\bar\xi$ leaves the diamond $\bar I_1=I(\bar\gamma(0),\bar\gamma(T_1))$ after the parameter $c$ i.e. $\bar\xi(c)\ll\bar\gamma(T_1)$. Here we are using that $c$ does not actually depend on $k$ as long as $k$ is negative. By \Cref{rem:curvaturebound}\ref{rem:curvaturebound2} the curvature bound now implies $\xi(c)\ll\gamma(T_1)$ and in particular $\xi(c)\in I_1$, which is what we wanted to show.
\end{proof}






We now observe that if $k=0$ one can explicitly write down a sharp constant $C$ in \Cref{thm:main1} for aligned chronological diamonds.


\begin{lemma}\label{lem:sharpconstant}
Let $X$ be a globally hyperbolic, timelike non-branching, regular measured Lorentzian length space. Assume that $X$ and its causally reversed  structure satisfy the $\TMCP^e(K,N)$-condition, and that $X$ satisfies the global $\TCBA(0)$-condition. 
Let $I_1 = I(x_1,y_1)$ and $I_2 = I(x_2,y_2)$ be such that $x_2\ll x_1\ll y_1\ll y_2$ and $x_1,x_2,y_1,y_2$ lie on a timelike geodesic. Then
\[
\frac{\mfm(I_2)}{\mfm(I_1)}\le \begin{cases}
\left(\frac{T_2}{T_1}\right)^N & \text{if }\ K\geq 0,\\
\left(\frac{\mfs_{K/(N-1)}(T_2)}{\mfs_{K/(N-1)}(T_1)}\right)^{N} & \text{if }\ K<0
\end{cases}
\]
where $T_i=\tau(x_i,y_i)$.
\end{lemma}
\begin{proof}
By the proof of \Cref{prop:2}, we know that $X$ satisfies the $(\rho,\rho^{-1},\Lambda)$-condition for any $\rho>1$ and $\Lambda \in (0,\infty]$.

Let $\mfq$, $Q$, $X_\alpha$, $\gamma_\alpha$, and $d_\alpha$ given by \Cref{thm:needle-decomposition} applied to $\mfm\llcorner_{I^+(x_1)}$. By \Cref{lem:disintegration}, for $\mfq$-a.e.\ $\alpha\in Q$ there exist $0\leq t_{\alpha,1}\leq t_{\alpha,2}\leq d_\alpha$ such that $\gamma_\alpha([0,t_{\alpha,i}])=\overline{X}_\alpha\cap J(x_1,y_i)$, $i=1,2$. Then, if $\triangle \bar{x}_1\overline{\gamma_\alpha(t_{\alpha,2})}\bar{y}_2$ is a comparison triangle for $\triangle x_1\gamma_\alpha(t_{\alpha,2})y_2$ in the Minkowski plane and $\overline{\gamma_\alpha(t_{\alpha,1})}\in [\bar{x}_1\overline{\gamma_\alpha(t_{\alpha,2})}]$ is a comparison point for $\gamma_\alpha(t_{\alpha,1})$, the intercept theorem implies 
\[
\frac{t_{\alpha,2}}{t_{\alpha,1}}=\frac{\tau(\bar{x}_1,\overline{\gamma_\alpha(t_{\alpha,2})})}{\tau(\bar{x}_1,\overline{\gamma_\alpha(t_{\alpha,1})})}\leq \frac{\tau(\bar{x}_1,\bar{y}_2)}{\tau(\bar{x}_1,\bar{y}_1)} = \frac{\tau(x_1,y_2)}{\tau(x_1,y_1)}.
\]
Hence we know that $X$ satisfies the $(\rho,\rho^{-1},\Lambda)$-condition with $\rho = \tau(x_1,y_2)/\tau(x_1,y_1)$, $\Lambda = \tau(x_1,y_2)$.

Let us first consider the case $K<0$ and let $F_{K,N,\rho^{-1}}(\alpha)$ be defined as in \eqref{equation1}. By a change of variables,
\begin{equation}\label{eq:change of variables F}
F_{K,N,\rho^{-1}}(\alpha) = \frac{\int_{0}^{t_{\alpha,2}}\mfs_{K/(N-1)}(t)^{N-1} dt}{\int_{0}^{\rho^{-1}t_{\alpha,2}}\mfs_{K/(N-1)}(t)^{N-1} dt}= \rho\cdot \frac{\int_{0}^{t_{\alpha,2}}\mfs_{K/(N-1)}(t)^{N-1} dt}{\int_{0}^{t_{\alpha,2}}\mfs_{K/(N-1)}(\rho^{-1}t)^{N-1} dt}.
\end{equation}
On the other hand, if 
$f(t) = \frac{\mfs_{K/(N-1)}(t)}{\mfs_{K/(N-1)}(\rho^{-1}t)}$, then 
\[
f'(t) = \frac{\mfc_{K/(N-1)}(t)\mfs_{K/(N-1)}(\rho^{-1}t) - \rho^{-1}\mfc_{K/(N-1)}(\rho^{-1}t)\mfs_{K/(N-1)}(t)}{\mfs_{K/(N-1)}(\rho^{-1}t)^2},
\]
which one can easily check is non-negative since the map $t\mapsto \tanh\left(\sqrt{\frac{-K}{N-1}}t\right)$ is concave. In particular, since 
$\rho=\tau(x_1,y_2)/\tau(x_1,y_1)$ and 
$t_{\alpha,2}\leq \tau(x_1,y_2)$ for $\mfq$-a.e. $\alpha\in Q$, it follows that
\begin{equation*}
\frac{\mfs_{K/(N-1)}(t)}{\mfs_{K/(N-1)}(\rho^{-1}t)} \leq \frac{\mfs_{K/(N-1)}(\tau(x_1,y_2))}{\mfs_{K/(N-1)}(\tau(x_1,y_1))}
\end{equation*}
for $\mfq$-a.e. $\alpha\in Q$ and $t\in [0,t_{\alpha,2}]$. Therefore,
\begin{equation}\label{eq:bound for quotient of generalized sin}
\mfs_{K/(N-1)}(t)^{N-1} \leq \left(\frac{\mfs_{K/(N-1)}(\tau(x_1,y_2))}{\mfs_{K/(N-1)}(\tau(x_1,y_1))}\right)^{N-1}\mfs_{K/(N-1)}(\rho^{-1}t)^{N-1}
\end{equation}
for $\mfq$-a.e. $\alpha\in Q$ and $t\in [0,t_{\alpha,2}]$. For $\alpha\in Q$ satisfying \eqref{eq:bound for quotient of generalized sin}, integrating over $[0,t_{\alpha,2}]$ and plugging into \eqref{eq:change of variables F} yields 
\[
F_{K,N,\rho^{-1}}(\alpha) \leq 
\rho\cdot\left(\frac{\mfs_{K/(N-1)}(\tau(x_1,y_2))}{\mfs_{K/(N-1)}(\tau(x_1,y_1))}\right)^{N-1} = \frac{\tau(x_1,y_2)}{\tau(x_1,y_1)}\cdot\left(\frac{\mfs_{K/(N-1)}(T_2)}{\mfs_{K/(N-1)}(T_1)}\right)^{N-1}.
\]
An argument analogous to the proof of \Cref{lem:disintegration} implies that
\begin{equation}\label{eq:case k = 0,K<0}
\frac{\mfm(I(x_1,y_2))}{\mfm(I(x_1,y_1))} \leq \frac{\tau(x_1,y_2)}{\tau(x_1,y_1)}\cdot\left(\frac{\mfs_{K/(N-1)}(T_2)}{\mfs_{K/(N-1)}(T_1)}\right)^{N-1}.
\end{equation}
The conclusion follows after repeating this argument with the causal diamonds $I(x_1,y_2)$ and $I(x_2,y_2)$ to obtain an upper bound for $\frac{\mfm(I(x_2,y_2))}{\mfm(I(x_1,y_2))}$; multiplying the resulting inequality with \eqref{eq:case k = 0,K<0}; and finally using the monotonicity of the map $t\mapsto \frac{\mfs_{K/(N-1)}(t)}{t}$.

Finally, if $K\geq 0$, then the $\TMCP^e(K,N)$-condition implies the $\TMCP^e(0,N)$-condition, and it is straightforward to verify that 
\[
F_{0,N,\rho^{-1}}(\alpha) =\rho^N = \left(\frac{\tau(x_1,y_2)}{\tau(x_1,y_1)}\right)^N.
\]
Following the proof of \Cref{lem:disintegration}, we obtain
\[
\frac{\mfm(I(x_1,y_2))}{\mfm(I(x_1,y_1))} \leq \left(\frac{\tau(x_1,x_2)}{\tau(x_1,y_1)}\right)^N,\quad
\frac{\mfm(I(x_2,y_2))}{\mfm(I(x_1,y_2))} \leq \left(\frac{\tau(x_2,x_2)}{\tau(x_1,y_2)}\right)^N,
\]
and multiplying these inequalities we obtain the desired result.
\end{proof}

\Cref{prop:1} and \Cref{prop:2} together imply that a globally hyperbolic, timelike non-branching, regular measured Lorentzian length space, which satisfies both the $\TMCP^e(K,N)$-condition and a global sectional curvature bound from above, has the aligned $(\rho,\Lambda)$-enlargement property. Now we get rid of the assumption that all diamonds must be aligned and arrive at a more precise version of \Cref{thm:main1}.


\begin{theorem}\label{thm:generaldoubling}
    Let $\mathcal X$ be a class of globally hyperbolic, timelike non-branching, regular measured Lorentzian length spaces $X$ such that $X$ and its causally reversed structure satisfy the $\TMCP^e(K,N)$-condition. Further suppose that, for some $k\in\mathbb R$, all $X\in\mathcal X$ satisfy the global $\TCBA(k)$-condition. If $k<0$ fix any constant $\Lambda\leq\frac{\pi}{2\sqrt{-k}}$ and if $K<0$ fix any constant $\Lambda<\infty$, otherwise set $\Lambda=\infty$.
    
    Then for all $\rho>1$ there exists $C=C(K,N,k,\rho,\Lambda)>0$ such that all $X\in\mathcal X$ have the $(\rho,\Lambda)$-enlargement property with constant $C$.
\end{theorem}

\begin{proof}
    Pick some $X\in\mathcal X$, then we perform all calculations in $X$ noting that all of the constants we obtain only depend on $K,N,k,\rho$ and $\Lambda$, but not on the choice of space. We will prove the Theorem in two steps. First we will find a constant $C$ such that if $\diam^\tau (I_2)\leq\Lambda$ and $I_2=I(x_2,y_2)$ is a $\rho$-enlargement of $I_1=I(x_1,y_1)$ with one coinciding endpoint (i.e. either $x_1=x_2$ or $y_1=y_2$) then $\mfm(I_2)\leq C\, \mfm(I_1)$. Without loss of generality we assume $x_1=x_2=:x$. We define the timelike geodesic $\gamma:[0,1]\to X$ as the unique distance realiser from $x$ to $y_2$. Note that in a space with a global timelike curvature bound from above, such curves are indeed unique by \cite[Proposition 4.8]{beran-napper-rott}. As in previous proofs we can define a parameter $T_1$ as the infimum over all $t\in[0,1]$ such that $\tau(\gamma(t),y_1)=0$. We claim that there exists a parameter $T_0>0$ only depending on $k$ and $\Lambda$ such that $T_0\leq T_1$ i.e. it is not possible to find diamonds $I_1$ and $I_2$ as above such that $T_1$ becomes arbitrarily small. 
    If this claim holds then $I_0:=I(x,\gamma(T_0))\subset I_1\subset I_2$ and in particular $I_2$ is an aligned $T_0^{-1}$-enlargement of $ I_0$. By \Cref{prop:2} we know there exists a $c$, only depending on $T_0^{-1},k$ and $\Lambda$, such that $X$ satisfies the $(T_0^{-1},c,\Lambda)$-condition. Further by \Cref{prop:1} there exists a constant $C$ depending on $K,N, c$ and $\Lambda$ such that $X$ has the aligned $(T_0^{-1},\Lambda)$-enlargement property. In particular we have
    $$\mfm(I_2)\leq C\, \mfm(I_0)\leq C\, \mfm(I_1),$$
    which is what we wanted to show. To finish the first step it remains to prove the claim, that is to find such a $T_0$. Once again it will be enough to find such a $T_0$ in the appropriate comparison space. Indeed take $I(\bar\gamma(0),\bar y_1)=:\bar I_1\subset\bar I_2:=I(\bar\gamma(0),\bar\gamma(1))$ in $\mathbb L^2(k)$ such that $\diam^\tau (\bar I_i)=\diam^\tau (I_i)$ for $i=1,2$, $\bar\tau(\bar\gamma(0),\bar y_1)=\tau(\gamma(0),y_1)$ and $\bar\tau(\bar y_1,\bar\gamma(1))=\tau(y_1,\gamma(1))$. Then $\triangle\bar\gamma(0)\bar y_1\bar\gamma(1)$ is a comparison triangle for $\triangle\gamma(0)y_1\gamma(1)$. If we have already found a $T_0>0$ such that $\bar\gamma(T_0)\leq \bar y_1$ then by \Cref{rem:curvaturebound}\ref{rem:curvaturebound2} the curvature bound implies that $\gamma(T_0)\leq y_1$ as well, so clearly this $T_0$ satisfies $T_0\leq T_1$. The existence of $T_0$ for the spaces $\mathbb L^2(k)$ is proven in \Cref{lem:generaldoubling} below. 

\begin{figure}[h]
\centering
\begin{tikzpicture}[line cap=round,line join=round,x=0.5cm,y=0.5cm]
\clip(-9.,-1.) rectangle (9.,17.);
\fill[line width=2.pt,fill=black,fill opacity=0.10000000149011612] (5.,7.) -- (0.,12.) -- (-3.,9.) -- (2.,4.) -- cycle;
\fill[line width=2.pt,fill=black,fill opacity=0.10000000149011612] (2.,4.) -- (7.,9.) -- (0.,16.) -- (-5.,11.) -- cycle;
\draw [line width=2.pt] (0.,0.)-- (8.,8.);
\draw [line width=2.pt] (8.,8.)-- (0.,16.);
\draw [line width=2.pt] (0.,16.)-- (-8.,8.);
\draw [line width=2.pt] (-8.,8.)-- (0.,0.);
\draw [line width=2.pt] (5.,7.)-- (0.,12.);
\draw [line width=2.pt] (0.,12.)-- (-3.,9.);
\draw [line width=2.pt] (2.,4.)-- (7.,9.);
\draw [line width=2.pt] (7.,9.)-- (0.,16.);
\draw [line width=2.pt] (0.,16.)-- (-5.,11.);
\draw [line width=2.pt] (-5.,11.)-- (2.,4.);
\draw[color=black] (1,8) node {$I_1$};
\draw[color=black] (3.5,10.5) node {$\tilde I_2$};
\draw[color=black] (-2,5) node {$I_2$};
\draw [fill=black] (0,0) circle (2.5pt);
\draw[color=black] (0,-0.6) node {$x_2$};
\draw [fill=black] (0,16) circle (2.5pt);
\draw[color=black] (0,16.6) node {$y_2$};
\draw [fill=black] (0,12) circle (2.5pt);
\draw[color=black] (0,12.6) node {$y_1$};
\draw [fill=black] (2,4) circle (2.5pt);
\draw[color=black] (2,3.4) node {$x_1$};
\end{tikzpicture}
\caption{Step 2 in the proof of \Cref{thm:generaldoubling}.}
\end{figure}
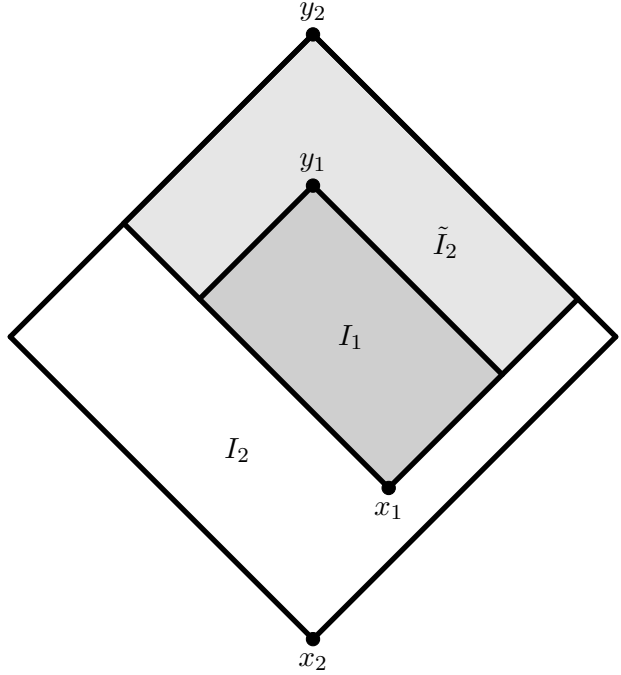
    In the second step we will show $\mfm(I_2)\leq C\, \mfm(I_1)$ for a general $\rho$-enlargement $I_1=I(x_1,y_1)$ of $I_2=I(x_2,y_2)$. To do this first define $\tilde I_2:=I(x_1,y_2)$. Since $\diam^\tau(\tilde I_2)\leq\diam^\tau (I_2)\leq\rho\, \diam^\tau (I_1)$ we know that $\tilde I_2$ is a $\rho$-enlargement of $I_1$, with the two diamonds sharing one endpoint. By the first step we already know that there exists a constant $\tilde C=\tilde C(K,N,k,\rho,\Lambda)$ such that 
    $$\mfm(\tilde I_2)\leq \tilde C\, \mfm(I_1).$$
    But now $\diam^\tau (I_2)\leq\rho\, \diam^\tau (I_1)\leq \rho\, \diam^\tau(\tilde I_2)$ so $I_2$ is a $\rho$-enlargement of $\tilde I_2$, with the two diamonds sharing one endpoint, hence with the same constant $\tilde C$ as before (note that both the diameter of $\tilde I_2$ and the diameter of $I_2$ are bounded above by $\Lambda$ so we can really choose the same constant twice) we obtain
    $$\mfm( I_2)\leq \tilde C\, \mfm(\tilde I_2)$$
    and altogether
    $$\mfm(I_2)\leq \tilde C^2\, \mfm(I_1).$$
    Therefore $C:=\tilde C^2$ works as our desired constant. 
\end{proof}

\begin{lemma}\label{lem:generaldoubling}
    Let $k\in\RR$, let $\rho>1$ and consider $\mathbb L^2(k)$. There exists a constant $T_0=T_0(k,\rho)$ such that the following holds:

    If $I_2:=I(x,y_2)$ is a $\rho$-enlargement of $I_1:=I(x,y_1)$ and $\gamma:[0,1]\to \mathbb L^2(k)$ is the unique timelike geodesic from $x$ to $y_2$ then for all $t<T_0$ we have $\gamma(t)\in I_1$. In words, $\gamma$ cannot leave the chronological diamond $I_1$ before time $T_0$.
\end{lemma}

\begin{proof}
    Let $I_1\subset I_2\subset \mathbb L^2(k)$ as above. Note that we only need to consider the case $\diam^\tau (I_2)=\rho\, \diam^\tau (I_1)$. If we already knew $T_0$ exists in this case and had $\diam^\tau (I_2)<\rho \diam^\tau (I_1)$ then there exists a diamond $\hat{I}_1:=I(x,\hat y_1)\subset I_1$ such that $\diam^\tau (I_2)=\rho\, \diam^\tau (\hat I_1)$. But now it clearly holds that $T_0\leq \hat T_1\leq T_1$, where $\hat T_1$ and $T_1$ are the parameters at which $\gamma$ leaves $\hat I_1$ and $I_1$ respectively. Henceforth we will call diamonds $I_1$, which share the starting point $x$ with $I_2$ and satisfy $\diam^\tau (I_2)=\rho\, \diam^\tau (I_1)$, admissible. Now we have to distinguish the cases $k=0$, $k>0$ and $k<0$.

\usetikzlibrary{arrows}

\definecolor{xdxdff}{rgb}{0.49019607843137253,0.49019607843137253,1.}
\definecolor{uuuuuu}{rgb}{0.26666666666666666,0.26666666666666666,0.26666666666666666}
\begin{figure}[h]
\centering
\begin{tikzpicture}[line cap=round,line join=round,>=triangle 45,x=4.0cm,y=4.0cm]
\clip(-1.4,1.8) rectangle (1.4,4.2);
\fill[line width=2.pt,fill=black,fill opacity=0.10000000149011612] (0.,2.) -- (1.,3.) -- (0.7516877853295644,3.2483122146704355) -- (-0.24831221467043552,2.2483122146704355) -- cycle;
\draw [samples=50,domain=-0.99:0.99,rotate around={90.:(0.,1.7210593199224706)},xshift=0.cm,yshift=6.86cm,line width=2.pt] plot ({1.2789406800775296*(1+(\x)^2)/(1-(\x)^2)},{1.1516752665290795*2*(\x)/(1-(\x)^2)});
\draw [line width=1.pt] (0.,2.)-- (1.,3.);
\draw [line width=2.pt] (1.,3.)-- (0.,4.);
\draw [line width=1.pt] (0.,4.)-- (-1.,3.);
\draw [line width=1.pt] (-1.,3.)-- (0.,2.);
\draw [line width=2.pt] (0.7516877853295644,3.2483122146704355)-- (-0.24831221467043552,2.2483122146704355);
\draw [line width=2.pt] (0.,2.)-- (0.,4.);
\draw [line width=1.pt] (0.,2.)-- (1.,3.);
\draw [line width=1.pt] (1.,3.)-- (0.7516877853295644,3.2483122146704355);
\draw [line width=1.pt] (0.7516877853295644,3.2483122146704355)-- (-0.24831221467043552,2.2483122146704355);
\draw [line width=1.pt] (-0.24831221467043552,2.2483122146704355)-- (0.,2.);
\begin{scriptsize}
\draw[color=black] (-1,3.65) node {$\tau_x^{-1}(\rho^{-1}T_2)$};
\draw [fill=black] (0.,2.) circle (2.5pt);
\draw[color=black] (0,1.94) node {$x$};
\draw [fill=black] (0.,4.) circle (2.5pt);
\draw[color=black] (0,4.0754357329670245) node {$y_2$};
\draw [fill=black] (0.7516877853295644,3.2483122146704355) circle (2.5pt);
\draw[color=black] (0.7821160419395771,3.3207354527672193) node {$y_1$};
\draw[color=black] (0.27,2.85) node {$\beta$};
\draw[color=black] (-0.07200583205900281,3.3704462496666077) node {$\gamma$};
\draw [fill=black] (0.,2.4966244293408715) circle (2.5pt);
\draw[color=black] (-0.11,2.570554335921904) node {$\gamma(T_0)$};
\draw[color=black] (0.5,3.6) node {$\eta$};

\end{scriptsize}
\end{tikzpicture}
\caption{Construction of $T_0$ for $k=0$}
\end{figure}
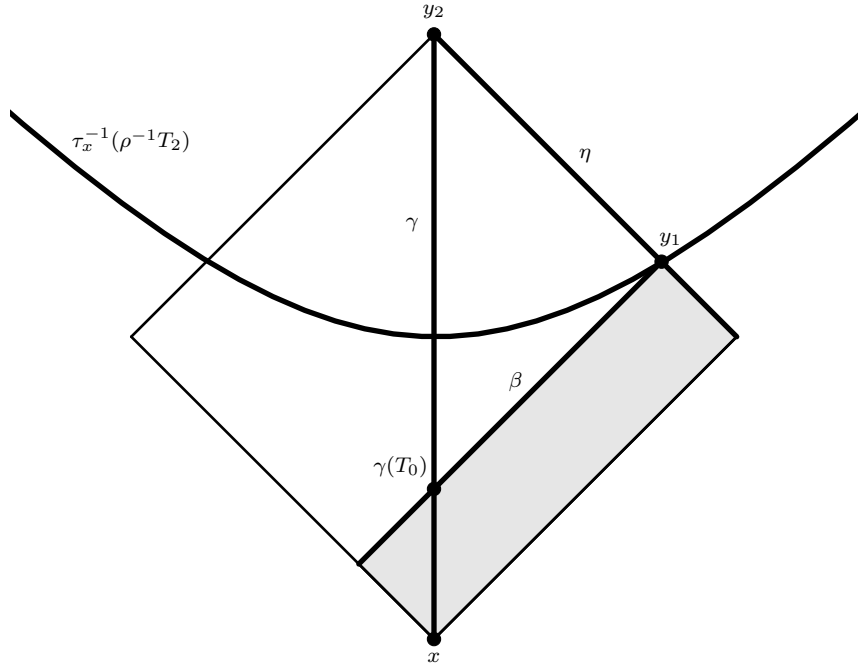
    
    Let $k=0$ so $\mathbb L^2(0)$ is the $2$-dimensional Minkowski space. Without loss of generality the curve $\gamma:[0,1]\to \mathbb L^2(0)$ is given by $t\mapsto (t\, T_2,0)$, where $T_2:=\diam^\tau(I_2)$ and in particular $x=(0,0)$ and $y_2=(T_2,0)$. Consider the set $\tau_x^{-1}(\{\rho^{-1}\, T_2\})$ which are all the points of time separation $\rho^{-1}\, T_2$ from $x$. The intersection of this set with the boundary of the diamond $I_2$ consists of two points. Choose one of them and call it $y_1$. By definition $y_1\gg x$ so we can define the parameter $T_0>0$ as the time at which $\gamma$ leaves $I(x,y_1)$. Clearly for all other admissible diamonds $I_1$, the parameter $T_1$ at which $\gamma$ leaves $I_1$ is at least $T_0$, so in that sense $T_0$ works as planned. We now calculate $T_0$ exactly, and show that it is independent of $T_2$, i.e. the size of $I_2$. We parametrize a part of the boundary of $I_2$ with the curve 
    $$\eta:[0,T_2]\to \mathbb L^2(0),\hspace{1cm}\eta(s):=\frac{1}{2}(T_2+t,T_2-t).$$
    The point $y_1$ defined above can be chosen to lie on $\eta$ and we want to find the value $s$ such that $\eta(s)=y_1$. In order to do this note that $\tau(x,\eta(s))=\sqrt{\left|\left(\frac{1}{2}(T_2-s)\right)^2-\left(\frac{1}{2}(T_2+s)\right)^2\right|}=\sqrt{T_2s}$, hence $s=\rho^{-2}T_2$ as $\rho^{-1}T_2=\tau(x,\eta(s))$. We can now calculate the point $\gamma(T_0)$ by intersecting $\gamma$ with a past directed null curve $\beta$ starting at $y_1$. Either by a direct computation or using the intercept theorem we see that $\gamma(T_0)=(T_2\rho^{-2},0)$ so (since $\gamma$ was defined on $[0,1]$) we obtain $T_0=\rho^{-2}$, which is indeed independent of $T_2$.
    

    For the case $k>0$ recall that $\mathbb L^2(k)$ satisfies the global $\TCBA(0)$-condition. Consider the causal triangle $\triangle xy_1y_2$ and a comparison triangle $\triangle\bar x\bar y_1\bar y_2$ in $\mathbb L^2(0)$. Using $T_0=\rho^{-2}$ as above we have $\bar\gamma(T_0)\leq \bar y_1$ and therefore by \Cref{rem:curvaturebound}\ref{rem:curvaturebound2} also $\gamma(T_0)\leq y_1$. In particular for $t<T_0$ we have $\gamma(t)\in I_1$ as desired.

    Finally let $k<0$. To begin with we assume that $T_2:=\diam^\tau (I_2)=\frac{\pi}{2\sqrt{-k}}$. Similar to before it is enough to consider the case where the  curve $\gamma$ is given by $\gamma(t)=(r\, \sin(t\, \frac{\pi}{2}),r\, \cos(t\, \frac{\pi}{2}),0)$, where $r=1/\sqrt{-k}$. As before define $y_1$ as one of the two intersections of $\tau_x^{-1}(\{\rho^{-1}\, T_2\})$ with the boundary of $I_2=I(\gamma(0),\gamma(1))$. Again we define $T_0$ as the parameter at which $\gamma$ leaves $I(x,y_1)$. First of all it is easy to see that for all other admissible diamonds $ I_1\subset I_2$, the parameter $T_1$ at which $\gamma$ leaves $ I_1$ is at least $T_0$. 


 \medskip
    {Now assume $\diam^\tau(I_2)<\frac{\pi}{2\sqrt{-k}}$. We claim that the same parameter $T_0$ still works. In order to see this pick $\bar k<k$ such that $\diam^\tau(I_2)=\frac{\pi}{2\sqrt{-\bar k}}$. Now define a timelike geodesic $\bar\gamma:[0,1]\to \mathbb L^2(\bar k)$ of the same length as $\gamma$ and define $\bar y_1$, $\bar I_1$ and $\bar I_2$ analogously to before. Since $\mathbb L^2(\bar k)$ can be seen as a conformal transformation of $\mathbb L^2(k)$ with constant conformal factor $\bar k/k$, the ratio of the length of the segment of $\bar \gamma$ such that it lies in $\bar I_1$ and the length of the segment of $\bar \gamma$ such that it does not lie in $\bar I_1$, is $T_0$. In particular $\bar\gamma(T_0)\leq \bar y_1$. Since $\mathbb L^2(k)$ satisfies the global $\TCBA(\bar k)$-condition by \Cref{rem:curvaturebound}\ref{rem:curvaturebound2} we also have $\gamma(T_0)\leq y_1$. Therefore $T_1\geq T_0$ as desired.} 
\end{proof}


In the positive-signature setting, the doubling condition implies bounds on the cardinality of $\eps$-separated sets. With this motivation in mind, we propose the following definition of $\eps$-separated sets in the Lorentzian setting.

\begin{definition}\label{def:eps-separated set}
Given $A\subset X$ and $\eps>0$, an \textit{$\eps$-separated set} in $A$ is a set of pairwise disjoint chronological diamonds $\{I(p_i,q_i)\}_{i\in \Omega}$ 
such that $I(p_i,q_i)\subset A$ and $\tau(p_i,q_i) = \eps$ for all $i\in\Omega$.
\end{definition}

We now state and prove a more precise version of \Cref{thm:main2}.

\begin{theorem}\label{thm:bound on eps separated sets}
Let 
$k,K\in\RR$, $N>1$ and $0<\eps<\Lambda<D_k/2$. Then there exists a constant
\[
\widetilde{C}= \widetilde{C}(k,K,N,\eps,\Lambda)>0
\]
such that the following holds. Let $(X,d,\ll,\leq,\tau,\mfm)$ be a globally hyperbolic, timelike non-branching, regular measured Lorentzian length space, where $X$ 
(and its causally reversed structure) satisfies the $\TMCP^e(K,N)$- and global $\TCBA(k)$-conditions. Moreover let $\tilde{I}$ 
be a chronological diamond with $\diam^\tau(\tilde{I})\leq \Lambda$. 
Then any $\eps$-separated set in $\tilde{I}$ has cardinality at most $\widetilde{C}$.
\end{theorem}

\begin{proof}
Let $\calF$ be an $\eps$-separated set in $\tilde{I}$, then
\begin{equation}\label{eq:epssep1}
\mfm(\tilde I) \geq \mfm\left(\bigcup_{I\in \calF} I\right) = \sum_{I\in\calF} \mfm(I).
\end{equation}
For any $I\in\mathcal F$ we have $\diam^\tau(I)=\varepsilon$, $I\subset \tilde I$ and $\diam^\tau(\tilde I)\leq \Lambda$. In particular $\tilde I$ is a $\Lambda/\varepsilon$ enlargement of $I$ so \Cref{thm:generaldoubling} implies 
\begin{equation}\label{eq:epsnet2}
\mfm(\tilde I)\leq \widetilde{C}\, \mfm(I)
\end{equation}
for some positive constant $\widetilde{C}=\widetilde C(K,N,k,\Lambda,\eps)\geq1$. 
Combining \eqref{eq:epssep1} and \eqref{eq:epsnet2} we obtain
\[\mfm(\tilde I)\geq \sum_{I\in\mathcal F}\mfm (I)\geq|\mathcal F|\, \widetilde C^{-1}\, \mfm(\tilde I).\]
Since $\mfm$ is a Radon measure with total support and chronological diamonds are open by strong causality, $\mfm(\tilde I)$ is positive and by compactness of causal diamonds finite. It follows that $|\mathcal F|\leq\widetilde C$. 
\end{proof}

Moreover, as in the positive-signature setting, it is possible to show that maximal $\varepsilon$-separated sets exist in the following sense.
\begin{lemma}\label{claim:maximal separated set}
Let $A\subset X$ and let $\varepsilon>0$ such that there exists a chronological diamond $I$
of timelike diameter $\varepsilon$ and $I\subset A$. Then there is a maximal $\eps$-separated set in $A$. 
\end{lemma}

\begin{proof}
It is clear that $\eps$-separated sets in $A$ are partially ordered by inclusion, and the union of a chain of $\eps$-separated sets in $A$ is $\eps$-separated in $A$. By Zorn's lemma, the claim follows.  
\end{proof}

In the metric setting, having uniform bounds for the cardinality of $\varepsilon$-separated sets on a given collection of compact metric spaces, along with a uniform bound on the diameter, yields precompactness of this collection with respect to Gromov--Hausdorff convergence. This follows from the fact that maximal $\varepsilon$-separated sets exist and they are $\varepsilon$-nets. 

In the Lorentzian setting this relation between $\varepsilon$-separated sets and $\varepsilon$-nets is not so immediate. The main problem is that maximal $\varepsilon$-separated sets of a Lorentzian pre-length space do not necessarily cover it, but the collection of diamonds with time separation $\leq \varepsilon$ that intersect at least one element of the $\varepsilon$-separated set does cover the space (cf. \Cref{prop:separated set yields net}). However, such a collection is rarely finite and it is not clear how to induce a net of bounded cardinality out of this construction.

This is the main motivation for considering a particular class of chronological diamonds, in a particular family of Lorentzian pre-length spaces, namely, generalized cones, for our Lorentzian precompactness theorems. We do this in the following section.

\section{Generalized cones}\label{ss:generalized cones bishop gromov}
In this section we only consider Lorentzian generalized cones of the form $X = \prescript{-}{}J\times_f Y$, where $(Y,d_Y)$ is a length space, $J\subset\RR$ is an open interval and $f\colon J\to (0,\infty)$ is a continuous function. The main motivation for this is to have a distinguished timelike direction at any given point, which in turn allows the choice of canonical chronological diamonds as well as canonical enlargements of such diamonds in our arguments. Although we expect this assumption to be unnecessary, we currently lack techniques to obtain a more general result.

Under our current assumptions, we define the following collection of chronological and causal diamonds:
\begin{equation}\label{eq:hat J}
\begin{split}
&\hat{\mathcal{I}}(X) = \{I((t,y),(t',y)): t\leq t',\ y\in Y\}\\
&\hat{\calJ}(X) = \{J((t,y),(t',y)): t\leq t',\ y\in Y\}.
\end{split}
\end{equation}
i.e. all chronological or causal diamonds such that their starting- and endpoint have the same spacial component. 

Working with diamonds in $\hat{\mathcal{I}}(X)$ and $\hat{\mathcal{J}}(X)$ makes some computations less cumbersome. For example, if $I((t,y),(t',y))\in \hat{I}(X)$ then, 
\[
\tau((t,y),(t',y)) = t'-t.
\]
Indeed, by considering the future-directed causal curve given by $\gamma_o(s) = ((1-s)t+st',y)$, it is clear that $\gamma_o(0) = (t,y)$, $\gamma_o(1) = (t',y)$ and 
\[
L(\gamma_o) = t'-t,
\]
whereas for any other future-directed causal curve $\gamma=(\alpha,\beta)\colon [a,b]\to X$ with $\gamma(a) = (t,y)$ and $\gamma(b) = (t',y)$, we have 
\[
L(\gamma) = \int_a^b \sqrt{\dot\alpha^2-(f\circ \alpha)^2v_\beta^2} \leq \int_a^b \dot\alpha = \alpha(b)-\alpha(a) = t'-t,
\]
which proves the claim.

We start by recalling some basic tools to work in generalized cones. The following is a description of the chronological relation, which, in our setting, is a straightforward application of the proof of \cite[Proposition 3.22]{alexander-graf-kunzinger-saemann2023}.

\begin{lemma}\label{lem:chronological relation generalized cones}
Let $X =  \prescript{-}{}J\times_f Y$ be a generalized cone. Then $(t,y)\ll (t',y')$ if and only if
\[
d_Y(y,y')< \int_{t}^{t'}\frac{1}{f(s)}ds.
\]
\end{lemma}

Next, the following lemma is a straightforward application of \cite[Lemma 4.1]{alexander-graf-kunzinger-saemann2023} to the particular case of causal diamonds in $\hat{\mathcal{J}}(X)$.

\begin{lemma}\label{lem:causal-diamond-bounds}
Let $X =  \prescript{-}{}J\times_f Y$ be a generalized cone, let $\tilde{I} = I((\bar{t}-r,\bar{y}),(\bar{t}+r,\bar{y}))\in\hat{\calI}(X)$ and assume that $m:=\inf f>0$. Then,
\[
\tilde{I} \subset \left\{ (t,y) \in X : 
\bar{t}-r < t < \bar{t}+r,\ 
y \in {B}^{Y}_{\frac{r}{m}}(\bar{y}) 
\right\},
\]
where ${B}^{Y}_{\delta}(\bar{y})$ denotes the open ball of radius $\delta$ centered at $\bar{y}$ in $Y$.
\end{lemma}

The lemma states that any chronological diamond in $\hat{\mathcal{I}}(X)$, with timelike diameter $2r$, is contained in a ``timelike cylinder'' of radius $r/m$ and height $2r$. It can easily be seen that the converse statement holds as well, i.e. for any timelike cylinder as above we can find a causal diamond centered at the same point that contains it.

\begin{lemma}\label{lem:timelike-cylinder-bounds}
    Let $X =  \prescript{-}{}J\times_f Y$ be a generalized cone, let $C:=\{(t,y)\in X : \bar t-r\leq t\leq \bar t+r,\ y\in \overline{B}^Y_R(\bar y)\}\subset X$ and assume that $M:=\sup f<\infty$. Then, for any $\varepsilon>0$ such that $\bar t\pm r\pm RM\pm\varepsilon\in J$, we have 
    $$C\subset J((\bar t-r-RM-\varepsilon,\bar y),(\bar t+r+RM+\varepsilon,\bar y)).$$
\end{lemma}

\begin{proof}
    It is enough to show that $(\bar t+r,y)\ll(\bar t+r+RM+\varepsilon,\bar y)$ for all $y\in Y$ such that $d_Y(y,\bar y)\leq R$. Indeed once we have proven this, we also have $(t,y)\ll (\bar t+r+RM+\varepsilon,\bar y)$ for all $t\leq \bar t+r$, since the curve with constant spacial component $t\mapsto(t,y)$ is timelike. Hence $C\subset J^-(\bar t+r+RM+\varepsilon,\bar y)$ and analogously $C\subset J^+(\bar t-r-RM-\varepsilon,\bar y)$. It remains to prove the claim. To do so we estimate
    \begin{align*}
        \int_{\bar t+r}^{\bar t+r+RM+\varepsilon}\frac{1}{f(s)}ds\geq\int_{\bar t+r}^{\bar t+r+RM+\varepsilon}\frac{1}{M}ds=(\bar t+r+RM+\varepsilon-\bar t-r)\frac{1}{M}>R\geq d(y,\bar y),
    \end{align*}
    and the claim follows from \Cref{lem:chronological relation generalized cones}.
\end{proof}

The following theorem summarizes parts of \cite[Lemma~3.13, Theorem~3.29]{alexander-graf-kunzinger-saemann2023} that we will use in the sequel.
\begin{theorem}\label{thm:nice parametrizations}
Let $(Y,d)$ be a geodesic space, let $X=\prescript{-}{}J\times_f Y$ be a generalized cone and let
\[
\gamma = (\alpha,\beta) \colon [0,b] \to X
\]
be a future-directed timelike geodesic. Then $\gamma$ admits an absolutely continuous parametrization with respect to $\tau$-arclength, i.e.,
\[
-\dot{\alpha}^2 + (f \circ \alpha)^2 \, v_\beta^2 = -1 \quad \text{a.e.}
\]
and with respect to such parametrization, $v_\beta$ is proportional to $\frac{1}{(f \circ \alpha)^2}$.

Moreover, $\gamma$ also admits a parametrization such that $\alpha(t)=t$.
\end{theorem}

\subsection{Generalized cones with timelike curvature bounds}\label{s:separated sets in generalized cones}

In this subsection we consider generalized cones satisfying the $\TMCP^e(K,N)$- and global $\TCBA(k)$-conditions for some $K,k\in \mathbb{R}$ and $N>1$, and prove that certain classes of subsets of these spaces are precompact.
First we observe that this class is non-trivial.

\begin{example}\label{examples of generalized cones}
The following construction yields globally hyperbolic, timelike non-branching, regular measured Lorentzian length spaces (actually, weighted Lorentzian manifolds) satisfying the $\TMCP^e(K,N)$-condition for some constants $K,N$ as well as the global $\TCBA(0)$-condition. In particular, these spaces fulfill the hypothesis of \Cref{thm:generaldoubling}.

Consider $X = \prescript{-}{}J\times_f Y^n$, where $f\colon J\to (0,\infty)$ is $K'$-convex, i.e.\ $f''\geq K'f$. 
Let $Y^n$ be a complete Riemannian manifold of dimension $n\geq2$ with sectional curvature bounded above by $K$, for 
\[K :=\inf\{K'f^2-(f')^2\} > -\infty.\]  
Then, by \cite[Corollary~5.4]{alexander-graf-kunzinger-saemann2023}, $X$ has the $\TCBA(K')$-condition. Moreover, since $Y$ is geodesic and proper (by the Hopf--Rinow theorem), then $X$ is a globally hyperbolic, regular Lorentzian length space \cite[Corollary~4.9 and Proposition~4.10]{alexander-graf-kunzinger-saemann2023}, and assuming that $Y$ and $f$ are sufficiently smooth, we can guarantee that the Lorentzian metric in $X$ is $C^{1,1}$-smooth, which is sufficient to imply the timelike non-branching property (by the Picard--Lindel\"of theorem).

By \cite{harris1982}, Lorentzian manifolds with a one-sided timelike curvature bound have either constant curvature or unbounded curvature in the opposite direction. In this case, it is clear that $X$ does not have constant sectional curvature in general (see computations below), therefore it has unbounded timelike sectional curvature from below. Moreover, since $X$ has timelike curvature bounded from above, it has timelike Ricci curvature bounded above, and by \cite{dajczer-nomizu80}, $X$ also has timelike Ricci curvature unbounded from below.

Nevertheless, by imposing additional conditions on $f$, it is possible to define $F\colon X\to (0,\infty)$ such that $X$ equipped with the reference measure $\mfm=e^{-F}d\vol_{g_X}$ satisfies the $\TMCP^e(0,N)$-condition for sufficiently large $N$. 

Indeed, let $F\colon X\to \mathbb{R}$ be such that $\langle \nabla F,v\rangle = 0 $ for all $v\in TY$. 
If $z \in TX$ is such that $\langle z, z\rangle = -1$, then $z$ can be written as
\[
 z = \alpha \partial_t + \beta v 
\]
for some $v\in TY$ such that $|v|_Y = 1$, and $\alpha,\beta\in \mathbb{R}$ such that $\alpha^2-f^2\beta^2 = 1$. By straightforward computations (see, for example, \cite[Corollary 7.43]{oneill}), we obtain
\begin{align}
&\Ric_X(z,z) = 
-\alpha^2n\frac{f''}{f}+\beta^2\left(\Ric_Y(v,v)-\frac{f''}{f}-(n-1)\left(\frac{f'}{f}\right)^2\right), \label{eq:ricci example}\\
&\Hess(F)(z,z) = \alpha^2F''+\beta^2F'\frac{f'}{f}, \nonumber\\
&dF(z) = 
\alpha F'. \nonumber
\end{align}
Therefore, 
\begin{align*}
&\left(\Ric_X + \Hess(F) - \frac{1}{N-n}dF\otimes dF \right)(z,z) \\
&\quad = \alpha^2\left(-n\frac{f''}{f}+F''-\frac{(F')^2}{N-n}\right)+\beta^2\left(\Ric_Y(v,v) - \frac{f''}{f} - (n-1)\left(\frac{f'}{f}\right)^2+F'\frac{f'}{f}\right). 
\end{align*}
In particular, if
\begin{equation}\label{eq:conditions-bakry-emery-bounded below}
\begin{cases}
\displaystyle F'' \geq n\frac{f''}{f}+\frac{(F')^2}{N-n},\\
\displaystyle \Ric_Y(v,v) \geq \frac{f''}{f} + (n-1)\left(\frac{f'}{f}\right)^2 - F'\frac{f'}{f},
\end{cases}
\end{equation}
then
\[
\Ric^{N,F}_X(z,z)=\left(\Ric_X + \Hess(F) - \frac{1}{N-n}dF\otimes dF\right)(z,z) \geq 0
\] 
for any unit timelike $z$. By \cite[Theorem~3.1 and Proposition~3.12]{cavalletti-mondino2024}, it follows that $X$, endowed with $\mfm$, satisfies the $\TMCP^e(0,N)$-condition. 

We can achieve this, for example, by letting $I = (\varepsilon,1)$, $f(t) = t^2$, $K'=0$, $K=-4$, and $F(t,x)=-2n\log(t)+e^{Ct}$. In this case, \eqref{eq:conditions-bakry-emery-bounded below} is equivalent to
\[
\begin{cases}
\displaystyle  N \geq n+\frac{1}{C^2e^{Ct}}\left(Ce^{Ct}-\frac{2n}{t}\right)^2 \\
\displaystyle \Ric_Y(v,v) \geq \frac{8n-2}{t^2} - \frac{2}{t}Ce^{Ct}
\end{cases}
\]
which holds for sufficiently large $C, N > 0$ and suitable $Y$. Finally note that this space even satisfies the global $\TCBA(0)$-condition. Indeed it is easy to check that $X$ satisfies all assumptions of \cite[Corollary 4.9]{eroes-gieger} yielding the curvature bound does indeed hold globally. 
\end{example}

In the following we will try to relate the notions of $\eps$-separated sets and $\eps$-nets as defined in \cite[Definition 3.2]{mondino-saemann2025}. Recall that, given $\eps>0$ and a subset $A\subset X$, an \textit{$\eps$-net for $A$} is a collection $\{J(p_i,q_i)\}_{i\in \Omega}$ of causal diamonds in $X$ such that:
\begin{itemize}
    \item $\tau(p_i,q_i) \leq \eps$ for all $i\in \Omega$,
    \item $A\cap J(p_i,q_i)\neq\varnothing$ for all $i\in\Omega$, and
    \item $A\subset \bigcup_{i\in\Omega} J(p_i,q_i)$.
\end{itemize}

Recall that, in the positive-signature setting, if $(X,d)$ is a metric space, a set $S$ is $\eps$-separated if $d(x,y)\geq \varepsilon$ for all $x,y\in S$, whereas an $\eps$-net is a set of points $S\subset X$ such that $X\subset \bigcup_{x\in S} B_\eps(x)$. In that context, maximal $\eps$-separated sets are also $\eps$-nets.  Indeed if $S$ is a maximal $\eps$-separated set in $X$, then any other point in $X$ has to be at distance less than $\eps$ to some point in $S$, meaning 
$\{B_{\eps}(x)\}_{x\in S}$ covers $X$, i.e. $S$ is an $\eps$-net. Furthermore, if $X$ is a length space, it is not difficult to see that $S$ is an $\eps$-separated set if and only if $\{B_{\eps/2}(x)\}_{x\in S}$ is a disjoint set of balls. Therefore, in the context of length spaces, maximal disjoint sets of open $\eps/2$-balls yield $\eps$-nets simply by doubling the radii of those balls.  


In the Lorentzian setting this is not quite as simple. Let $\{I_i\}_{i\in\Omega}\subset\hat{\calI}(X)$ be a maximal $\eps$-separated set. The existence of such a maximal set of elements in $\hat{\mathcal{I}}(X)$ can be proven analogously to \Cref{lem:maximal separated set is net}. Then, denoting a doubling of each $I_i$ by $\widetilde I_i$, the set $\{\widetilde I_i\}_{i\in\Omega}$ is not necessarily a $2\eps$-net as illustrated by the following example.

\definecolor{qqwwzz}{rgb}{0.,0.4,0.6}
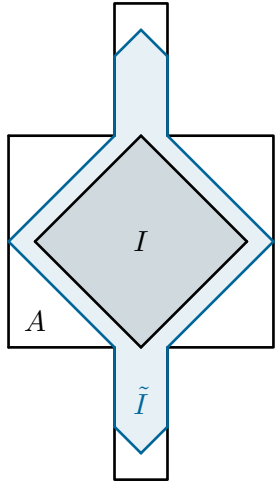
\begin{figure}[h]
\begin{tikzpicture}[line cap=round,line join=round,x=0.7cm,y=0.7cm]
\clip(-3.5,-5.) rectangle (3.5,5.);
\fill[line width=1.pt,fill=black,fill opacity=0.10000000149011612] (0.,-2.) -- (2.,0.) -- (0.,2.) -- (-2.,0.) -- cycle;
\fill[line width=1.pt,color=qqwwzz,fill=qqwwzz,fill opacity=0.10000000149011612] (0.,-4.) -- (0.5,-3.5) -- (0.5,-2.) -- (2.5,0.) -- (0.5,2.) -- (0.5,3.5) -- (0.,4.) -- (-0.5,3.5) -- (-0.5,2.) -- (-2.5,0.) -- (-0.5,-2.) -- (-0.5,-3.5) -- cycle;
\draw [line width=1.pt] (-2.5,2.)-- (-0.5,2.)-- (-0.5,4.5)-- (0.5,4.5)-- (0.5,2.)-- (2.5,2.)-- (2.5,-2.)-- (0.5,-2.)-- (0.5,-4.5)-- (-0.5,-4.5)-- (-0.5,-2.)-- (-2.5,-2.);
\draw [line width=1.pt] (-2.5,-2.)-- (-2.5,2.);
\draw [line width=1.pt] (0.,-2.)-- (2.,0.);
\draw [line width=1.pt] (2.,0.)-- (0.,2.);
\draw [line width=1.pt] (0.,2.)-- (-2.,0.);
\draw [line width=1.pt] (-2.,0.)-- (0.,-2.);
\draw [line width=1.pt,color=qqwwzz] (0.,-4.)-- (0.5,-3.5);
\draw [line width=1.pt,color=qqwwzz] (0.5,-3.5)-- (0.5,-2.);
\draw [line width=1.pt,color=qqwwzz] (0.5,-2.)-- (2.5,0.);
\draw [line width=1.pt,color=qqwwzz] (2.5,0.)-- (0.5,2.);
\draw [line width=1.pt,color=qqwwzz] (0.5,2.)-- (0.5,3.5);
\draw [line width=1.pt,color=qqwwzz] (0.5,3.5)-- (0.,4.);
\draw [line width=1.pt,color=qqwwzz] (0.,4.)-- (-0.5,3.5);
\draw [line width=1.pt,color=qqwwzz] (-0.5,3.5)-- (-0.5,2.);
\draw [line width=1.pt,color=qqwwzz] (-0.5,2.)-- (-2.5,0.);
\draw [line width=1.pt,color=qqwwzz] (-2.5,0.)-- (-0.5,-2.);
\draw [line width=1.pt,color=qqwwzz] (-0.5,-2.)-- (-0.5,-3.5);
\draw [line width=1.pt,color=qqwwzz] (-0.5,-3.5)-- (0.,-4.);
\draw[color=black] (0,0) node {$I$};
\draw[color=qqwwzz] (0,-3) node {$\tilde I$};
\draw[color=black] (-2,-1.5) node {$A$};
\end{tikzpicture}
\caption{$\varepsilon$-separated sets do not yield $2\varepsilon$-nets.}
\label{figure:nets}
\end{figure}

\begin{example}
    Consider an open subset $A$ of Minkowski space as as illustrated by \Cref{figure:nets}. The set $\{I\}$ consisting only of one chronological diamond $I$ of diameter $\varepsilon$ is a maximal $\varepsilon$-separated set. Its doubling $\tilde I$ however does not yield a $2\varepsilon$-net. 
\end{example}

To fix the issues brought up by the above example, we will have to define a different kind of doubling. Compare the following definition with \cite[Definition 4.1]{mccann-saemann2022}. 

\begin{definition}
Given $I(p,q)\in \hat{\calI}(X)$, define the set 
\[
D(I(p,q)) := \bigcup\{I(x,y)\in\hat{\calI}(X):I(x,y)\cap I(p,q)\neq \varnothing,\ \tau(x,y)=\tau(p,q)\}.
\]
\end{definition}

The lemma below is now analogous to the previously mentioned fact that maximal $\eps$-separated sets yield $\eps$-nets in the context of metric spaces. 
\begin{lemma}\label{lem:maximal separated set is net}
Let $\{I(p_i,q_i)\}_{i\in\Omega}\subset \hat{\calI}(X)$ be a maximal $\eps$-separated set in $A$. Then the collection $\{D(I(p_i,q_i))\}_{i\in \Omega}$ covers $A$.
\end{lemma}
\begin{proof}
Assume there exists $x\in A\setminus \bigcup_{i\in\Omega} D(I(p_i,q_i))$, then take any chronological diamond $I(p,q)\in \hat{\calI}(X)$ such that $x\in I(p,q)$ and $\tau(p,q)=\eps$. By definition of the operator $D$, it follows that $I(p,q)\cap I(p_i,q_i)=\varnothing$ for all $i\in \Omega$. Therefore, $\{I(p_i,q_i)\}_{i\in\Omega}\cup\{I(p,q)\}$ is an $\eps$-separated set in $A$, which contradicts the maximality of $\{I(p_i,q_i)\}_{i\in\Omega}$. 
\end{proof}

While at first glance this seems to resolve the problem of relating $\eps$-separated sets with $\eps$-nets, note that the sets $D(I(p_i,q_i))$ are neither chronological nor causal diamonds. To resolve this issue we will try to find a causal diamond $\tilde J_i\in\hat\calJ(X)$ of timelike diameter $\rho\varepsilon$ such that $D(I(p_i,q_i))\subset \tilde J_i$. Then the sets $\tilde J_i$ should give a $\rho\eps$-net. To make this more precise we will introduce the following notion.

\begin{definition}\label{def:homothety}
Let $I(p,q)\in\hat{\mathcal{I}}(X)$ be such that $p=(t-r,y)$, $q=(t+r,y)$ with appropriate $t\in\RR,r>0$, and let $\rho\geq 0$. Then the \textit{canonical $\rho$-enlargement} of $I(p,q)$, denoted by $F_\rho(I(p,q))$, is the causal diamond $J(p',q')$ given by
\[p':=(t-\rho r,y),\quad q':=(t+\rho r,y).\]
\end{definition}

\begin{remark}
\begin{enumerate}[label=(\roman*)]
    \item The canonical $\rho$-enlargement of a chronological diamond is defined to be a causal diamond because we use them to construct $\varepsilon$-nets, which in turn are defined as a collection of causal diamonds. 
    \item Note that given a chronological diamond $I_o=I((t-r,y),(t+r,y))$ the object $F_\rho(I_o)$ is not always defined since the parameters $t-\rho r$ or $t+\rho r$ might not lie in $J$, the domain of $f$. Going forward we will always have to take care of this subtlety. 
\end{enumerate}
\end{remark}

We are now positioned to prove that any maximal $\eps$-separated set yields an $\eps'$-net for certain $\eps'>\eps>0$.







\begin{proposition}\label{prop:separated set yields net}
    Let $A\subset X$ be nonempty, let $\varepsilon>0$, assume that $m:=\inf f>0$ and $M:=\sup f<\infty$ and define $\rho:=3\,(1+M/m)$. Let $\{I_i\}_{i\in\Omega}\subset \hat{\calI}(X)$ be a maximal $\varepsilon$-separated set in $A$. Then $\{F_\rho(I_i)\}_{i\in\Omega}$ is a $\rho\varepsilon$-net for $A$ if all causal diamonds $F_\rho(I_i)$ are defined.
\end{proposition}

\begin{proof}
    We need to show that $\{F_\rho(I_i)\}_{i\in\Omega}$ covers $A$. Recall that in \Cref{lem:maximal separated set is net} we have already proven that $\{D(I_i)\}_{i\in\Omega}$ possesses this property, so it suffices to show that $D(I_i)\subset F_\rho(I_i)$ for all $i\in\Omega$. We start by claiming that 
    $$D(I_i)\subset \left\{(t,y):t_i-\frac{3}{2}\varepsilon\leq t\leq t_i+\frac{3}{2}\varepsilon,\ y\in\overline{B}_{\frac{3}{2m}\varepsilon}(y_i)\right\}=:C_i,$$
    where $I_i=I((t_i-\varepsilon/2,y_i),(t_i+\varepsilon/2,y_i))$. Indeed if $(t,y)\in D(I_i)$ then there exists a chronological diamond $I_i'=I((t_i'-\varepsilon/2,y_i'),(t_i'+\varepsilon/2,y_i'))$ such that $(t,y)\in I_i'$ and $I_i'\cap I_i\neq\varnothing$. By \Cref{lem:causal-diamond-bounds} we have $|t-t_i'|\leq\varepsilon/2$ and $|t_i'-t_i|\leq \varepsilon$ so in particular $|t-t_i|\leq\frac{3}{2}\varepsilon$. Similarly we use \Cref{lem:causal-diamond-bounds} to show $d_Y(y,y_i)\leq\frac{3}{2m}\varepsilon$ and the claim follows. Now apply \Cref{lem:timelike-cylinder-bounds} to see that 
    $$C_i\subset J\left(\left(t_i-\frac{\varepsilon}{2}(3+3M/m),y\right),\left(t_i+\frac{\varepsilon}{2}(3+3M/m),y\right)\right)=F_\rho(I_i),$$
    which finishes the proof. 
\end{proof}

\Cref{prop:separated set yields net} will be used in two ways. First, to bound the Lorentzian Hausdorff dimension \cite{mccann-saemann2022} of a generalized cone in terms of the doubling constant and second, to prove precompactness of certain classes of Lorentzian length spaces with respect to the Lorentzian Gromov--Hausdorff convergence. For the former, we restrict ourselves to the case $J=\RR$ for simplicity and we first need the following analog of \cite[Thm.\ 4.12]{mccann-saemann2022}.
\begin{proposition}[Ratio of measures via doubling]\label{prop-rat-mea}
    Let $k,K\in \mathbb R$, $N>1$ and $0<\Lambda<D_k/2$. Let $X$ be a globally hyperbolic, timelike non-branching, regular generalized cone, satisfying $\diam^\tau(X)\leq \Lambda$ and the global $\TCBA(k)$-condition and it can be equipped with a non-negative, full support Radon measure $\mfm$ such that it (and its causally reversed structure) satisfies the $\TMCP^e(K,N)$-condition.
    Assume that $m:=\inf f>0$ and $M:=\sup f<\infty$ and define $\rho:=3\,(1+M/m)$. Let $I, I'\in \hat{\calI}(X)$ with $I\subseteq I'$, $\diam^\tau(I')\leq\Lambda$. 
    
    Then there is a constant $C'=C'(K,N,k,\rho,\Lambda)\geq 1$ such that
    \begin{align}\label{eq-prop-rat}
        \frac{\mfm(I)}{\mfm(I')}\geq C' \Bigl(\frac{\diam^\tau(I)}{\diam^\tau(I')}\Bigr)^\kappa,
    \end{align}
    where $\kappa:=\log_\rho(C)$ and $C=C(K,N,k,\rho,\Lambda)$ is given by  \Cref{thm:generaldoubling}.
\end{proposition}
\begin{proof}
    Let $I=I(p,q)$, $I'=(p',q')\in \hat{\calI}(X)$ with $I\subseteq I'$ and $p=(t-r,y), q=(t+r,y)$, $p'=(t'-r',y'),q'=(t'+r',y')$. $F_\rho(I)$ is a $(\rho,\Lambda)$-enlargement of $I$, so
    \begin{align}\label{eq-prop-rat-dou}
    \mfm(F_\rho(I))\leq C\, \mfm(I),
    \end{align}
    where $C$ be the constant given by \Cref{thm:generaldoubling}.
    
    From $I\subseteq I'$ we deduce that $|t'-t|\leq r'-r$ and $d(y,y')<\frac1m(r'-r)$ by \Cref{lem:chronological relation generalized cones}.

    Let $s\in\mathbb{N}$ 
    be minimal 
    such that $(2+\frac Mm)r'\leq (1+\frac Mm + \rho^s)r$. We claim that $I'\subseteq F_{\rho^s}(I)= J(p^s, q^s)$, where $p^s=(t-r\rho^s,y), q^s=(t+r\rho^s,y)$. By \Cref{lem:chronological relation generalized cones} it suffices to show 
    \begin{align*}
        d(y,y')\leq \min\Bigl(\int_{t-r\rho^s}^{t'-r'}\frac1f,\int_{t'+r'}^{t+r\rho^s}\frac1f\Bigr).
    \end{align*}
We estimate
\begin{align*}
    \int_{t-r\rho^s}^{t'-r'}\frac1f &\geq \frac1M (t'-r'-t+r\rho^s)\geq \frac1M((1+\rho^s)r-2r')\\
    & = \frac1m(r'-r) + \frac1M\underbrace{\left(\left(1 +\frac Mm+\rho^s\right)r - \left(2+\frac Mm\right)r'\right)}_{\geq 0} > d(y,y').
\end{align*}
Similarly,  $\int_{t'+r'}^{t+r\rho^s}\frac1f\geq \frac1M((1+\rho^s)r-2r')> d(y,y')$.
  
    
  Now, applying \eqref{eq-prop-rat-dou} $s$-times gives 
    \begin{align}\label{eq-rat-mea-dou}
    \mfm(I')\leq C^s \mfm(I).
    \end{align}
Also, the minimality of $s$ implies that\[\left(2+\frac Mm\right)r'> \left(1+\frac Mm + \rho^{s-1}\right)r>\rho^{s-1}r,\] and therefore $\frac{\diam^\tau(I)}{\diam^\tau(I')}< (2+\frac Mm) \rho^{-s+1}$. Setting $\kappa:=\log_\rho(C)$ gives 
\begin{align*}
    \left(\frac{\diam^\tau(I)}{\diam^\tau(I')}\right)^\kappa < \left(2+\frac Mm\right)^\kappa \rho^{(-s+1)\kappa} = \left(2+\frac Mm\right)^\kappa  C^{-s+1} \leq\left(2+\frac Mm\right)^\kappa  C \frac{\mfm(I)}{\mfm(I')}, 
\end{align*}
    by inequality \eqref{eq-rat-mea-dou}.
\end{proof}

\begin{corollary}[Bound on Lorentzian Hausdorff dimension]
    Let $k,K\in \mathbb R$, $N>1$ and $0<\Lambda<D_k/2$. Let $X$ be a globally hyperbolic, timelike non-branching, regular generalized cone, satisfying $\diam^\tau(X)\leq \Lambda$ and the global $\TCBA(k)$-condition and it can be equipped with a non-negative, full support Radon measure $\mfm$ such that it (and its causally reversed structure) satisfies the $\TMCP^e(K,N)$-condition.
    Assume that $m:=\inf f>0$ and $M:=\sup f<\infty$, and define $\rho:=3\,(1+M/m)$. 
    
    Then $\dim^\tau(X)\leq \kappa$, where $\kappa:=\log_\rho(C)$ and $C=C(K,N,k,\rho,\Lambda)$ is given by \Cref{thm:generaldoubling}.
\end{corollary}
\begin{proof}
    It suffices to show $\dim^\tau(I')\leq \kappa$ for each chronological diamond $I'=I(p',q')\in\hat{\calI}(X)$ with $\diam^\tau(I')\leq \Lambda$, by \cite[Lem.\ 3.5]{mccann-saemann2022} as any generalized cone is locally $D$-uniform, where $D$ is the product metric on $X=I\times Y$ (in fact $\tau\leq D$) and $X$ is Polish by assumption. From \Cref{prop-rat-mea} there is a constant $C''>0$ such that for every $I\in\hat{\calI}$ with $I\subseteq I'$ we have 
    \begin{align*}
    \mfm(I)\geq C'' \diam^\tau(I)^\kappa,
    \end{align*}
    by setting $C'':=C'\frac{\mfm(I')}{\diam^\tau(I')^\kappa}$. Here we use that $\mfm(I')<\infty$ by global hyperbolicity.

    At this point let $\delta>0$ and let $(I_i)_{i\in \Omega_\delta}$ be a maximal $\delta$-separated set in $I'$, then 
    \begin{align*}
    \mfm(I')\geq \mfm\left(\cup_{i\in\Omega_\delta}I_i\right) = \sum_{i\in\Omega_\delta}\mfm(I_i) \geq C'' |\Omega_\delta| \delta^\kappa\,,
    \end{align*}
    i.e., $|\Omega_\delta|\leq C''' \delta^{-\kappa}$ as $\mfm(I')\in(0,\infty)$. By \Cref{prop:separated set yields net} the family $(F_\rho(I_i))_{i\in\Omega_\delta}$ is a $\rho\delta$-net for $I'$. Moreover, $\diam^D(F_\rho(I_i))\leq \sqrt{1+\frac1m}\,\delta=:c\delta$ by \Cref{lem:causal-diamond-bounds}. Thus, we can estimate the Lorentzian Hausdorff measures
    \begin{align*}
    \mathcal{V}^\kappa_{c\delta}(I')\leq \omega_\kappa \sum_{i\in\Omega_\delta}\diam^\tau(F_\rho(I_i))^\kappa\leq \omega_\kappa |\Omega_\delta| \rho^\kappa \delta^\kappa \leq \omega_\kappa \rho^\kappa C''' < \infty.
    \end{align*}
    Then letting $\delta\searrow0$ yields $\mathcal{V}(I')\leq  \omega_\kappa \rho^\kappa C''' < \infty$, as required.
\end{proof}

Now we are ready to state and prove a more precise version of \Cref{thm:main3}.
\begin{theorem}\label{thm:bound on eps net}
Let $0<m\leq M<\infty$, $k,K\in\RR$, $N > 1$ and $0<\eps<\Lambda<D_k/2$. 
Then there exists a constant $\hat C=\hat C(m,M,k,K,N,\eps,\Lambda)>0$ such that the following implication holds.

Let $(X,d,\ll,\leq,\tau,\mfm)$ be a globally hyperbolic, timelike non-branching, regular measured Lorentzian length space, where $X=\prescript{-}{}J\times_f Y$ is a generalized cone, $m\leq f \leq M$, and $X$ (and its causally reversed structure) satisfies the $\TMCP^e(K,N)$- and global $\TCBA(k)$-conditions. Moreover let $\tilde{I}=I((s,y),(t,y))\in \hat{\calI}(X)$ be such that $\diam^\tau(\tilde{I})\leq \Lambda$ and $s,t$ are at least $\frac{3\varepsilon}{2}\left(1+\frac{M}{m}\right)$ away from the boundary of the interval $J$. Then there exists an $\eps$-net for $\tilde I$ of cardinality at most $\hat{C}$.
\end{theorem}

\begin{proof}
By \Cref{thm:bound on eps separated sets} and \Cref{claim:maximal separated set} we know that there exists a constant $\hat C$ only depending on $m,M,k,K,N,\eps, \Lambda$ such that $\tilde I$ contains a maximal $\rho^{-1}\varepsilon$-separated set of cardinality at most $\hat{C}$, where $\rho=3\left(1+\frac{M}{m}\right)$. Denote this set by $\{I_i\}_{1\leq i\leq \hat C}$. By \Cref{prop:separated set yields net} the set $\{F_\rho(I_i)\}_{1\leq i\leq \hat C}$ yields the desired $\varepsilon$-net. Note that the sets $F_\rho(I_i)$ exist because we required the endpoints of $\tilde I$ to be at least $\frac{3\varepsilon}{2}\left(1+\frac{M}{m}\right)$ away from the boundary of the interval $J$, so in particular the endpoints of the diamonds $F_\rho(I_i)$ will have time components within $J$.
\end{proof}

As an immediate consequence of \Cref{thm:bound on eps net}, we can finally prove the following more precise version of \Cref{thm:main4}.
\begin{theorem}\label{thm:diamond precompact}
For any $0<m\leq M<\infty$, $k,K\in\RR$, $N > 1$, $0<\Lambda<D_k/2$, the class of covered Lorentzian pre-length spaces $(X,\mathcal{U})$ given below is precompact with respect to the Lorentzian Gromov--Hausdorff convergence as defined in \cite[Definition 3.12]{mondino-saemann2025}: 
\begin{itemize}
\item $X=I((\bar s, \bar y),(\bar t,\bar y))$ is a subset of the globally hyperbolic, timelike non-branching, regular Lorentzian length space $\prescript{-}{}J\times_f Y$, which is a generalized cone with $m\leq f \leq M$, $\diam^\tau(X)\leq\Lambda$, the global $\TCBA(k)$-condition holds, and $X$ can be equipped with a non-negative, full-support Radon measure $\mfm$ such that it (and its causally reversed structure) satisfies the $\TMCP^e(K,N)$-condition.
\item $\mathcal{U}=\{U_n\}_{n\in\mathbb N}=\{I((\bar s+1/n,\bar y),(\bar t-1/n,\bar y))\}_{n\in \mathbb{N}}$ for some fixed $\bar y\in Y$.
\end{itemize}
\end{theorem}

\begin{proof}
    We need to check that points $(i)-(iii)$ from \cite[Theorem 6.2]{mondino-saemann2025} are satisfied.

    $\underline{(i)}:$ This holds trivially since by our assumptions all $X$ have their timelike diameter bounded from above by $\Lambda$.

    $\underline{(ii)}:$ Given $\varepsilon>0$ and $n\in\mathbb N$ we need to find a constant $C$ (independent of $X$) such that any $U_n\subset X$ admits an $\varepsilon$-net $S_\varepsilon^n$ of cardinality at most $C$. This is almost a direct consequence of \Cref{thm:bound on eps net} but we need to take care of the subtlety that the $\varepsilon$-nets constructed in \Cref{thm:bound on eps net} might consist of diamonds $J((t,y),(s,y))$, which are not a subset of the space $X$, if they are defined at all. To fix this issue we have to find some $\varepsilon_n$, such that any diamond in $\hat{\mathcal J}(X)$ of diameter $\varepsilon_n$, whose center lies in $U_n$, must be contained in $X$. By the center of a diamond we mean the midpoint of the geodesic connecting its endpoints. It is enough to consider these diamonds since $\varepsilon$-nets constructed in \Cref{thm:bound on eps net} consist of aligned enlargements of chronological diamonds which are entirely contained in $U_n$. Now we can choose $C$ to be the constant given by \Cref{thm:bound on eps net} for $\min\{\varepsilon,\varepsilon_n\}$. So let $X=I((\bar s,\bar y),(\bar t,\bar y))$ and $U_n=I((\bar s+1/n,\bar y),(\bar t-1/n,\bar y))$ and we claim that $\varepsilon_n:=\frac{1}{n}\frac{m}{M}$ suffices. It is enough to check that if $(t,y)\in U_n$ then $(t\pm\varepsilon_n,y)\in X$. Indeed if this is true then any causal diamond of diameter $\varepsilon_n$, with center $(t,y)$, must be contained in $I((t-\varepsilon_n,y),(t+\varepsilon_n,y))\subset X$.
    
    First note that $(t+\varepsilon_n,y)\gg(\bar s,\bar y)$ clearly holds. Now recall that by \Cref{lem:chronological relation generalized cones} $(t,y)\ll(\bar t-1/n,\bar y)$ is equivalent to 
    \begin{equation*}
        d_Y(y,\bar y)<\int_t^{\bar t-1/n}\frac{1}{f(s)} ds. 
    \end{equation*}
    Therefore to show $(t+\varepsilon_n,y)\ll(\bar t,\bar y)$ it is enough to prove 
    \begin{equation}\label{eq:integralinequality}
        \int_{t+\varepsilon_n}^{\bar t}\frac{1}{f(s)}ds\geq\int_t^{\bar t-1/n}\frac{1}{f(s)} ds.
    \end{equation}
    We now distinguish the cases $t<t+\varepsilon_n<\bar t-1/n<\bar t$ and $t<\bar t-1/n\leq t+\varepsilon_n<\bar t$. In the former \eqref{eq:integralinequality} is equivalent to 
    \begin{equation}\label{eq:integralinequality case 1}
        \int_{\bar t-1/n}^{\bar t}\frac{1}{f(s)} ds-\int_t^{t+\varepsilon_n}\frac{1}{f(s)}ds\geq0
    \end{equation}
    and indeed we can estimate the left hand side of \eqref{eq:integralinequality case 1} as follows:
    \begin{align*}
        \int_{\bar t-1/n}^{\bar t}\frac{1}{f(s)} ds-\int_t^{t+\varepsilon_n}\frac{1}{f(s)}ds\geq \frac{1}{n}\frac{1}{M}-\varepsilon_n\frac{1}{m}=\frac{1}{n}\frac{1}{M}-\frac{1}{n}\frac{1}{M}=0.
    \end{align*}
    Now consider the second case, where $\varepsilon_n\geq\bar t-1/n-t$. We start by estimating the length of the interval $[t+\varepsilon_n,\bar t]$:
    \begin{align*}
        \bar t-(t+\varepsilon_n)&=(\bar t-1/n-t)+(1/n-\varepsilon_n)\\
        &=(\bar t-1/n-t)+(M/m -1)\varepsilon_n\\
        &\geq(\bar t-1/n-t)+(M/m -1)(\bar t-1/n-t)=M/m\, (\bar t-1/n-t).
    \end{align*}
    This yields the following inequality for our integrals:
    \begin{align*}
        \int_{t+\varepsilon_n}^{\bar t}\frac{1}{f(s)}ds-\int_t^{\bar t-1/n}\frac{1}{f(s)} ds\geq (\bar t-(t+\varepsilon_n))\frac{1}{M}-(\bar t-1/n-t)\frac{1}{m}\geq0.
    \end{align*}
    This finishes the proof of the second case and hence our proof of $(ii)$.

    $\underline{(iii)}:$ This does not necessarily hold for the $\varepsilon$-nets constructed in $(ii)$, but by taking unions we can inductively construct larger $\varepsilon$-nets such that this property is satisfied. Indeed define 
    $$\tilde C(n,\varepsilon):=\sum_{i=1}^{n}C(i,\varepsilon)$$
    then the set 
    $$\tilde S_\varepsilon^n:=\bigcup_{i=1}^nS_\varepsilon^i$$
    is still an $\varepsilon$-net of $U_n$ and its cardinality is at most $\tilde C(n,\varepsilon)$, which is independent of $X$. 
\end{proof}

In the theorem above it might seem unnatural to only consider spaces that can be written as chronological diamonds in generalized cones. An alternative setting would be to consider generalized cones defined on the entire real line, and proving compactness of classes of such spaces by exhausting them with chronological diamonds of increasing diameter. In the theorem below we will do just that and the proof will indeed be simpler than the one of \Cref{thm:diamond precompact}. The problem of this approach is that it is unclear if such spaces, with non-constant curvature, even exist. Our approach in \Cref{examples of generalized cones} to construct generalized cones that satisfy the global $\TCBA(0)$-condition, have non-constant curvature and can be equipped with a measure such that they also satisfy the $\TMCP^e(K,N)$-condition, only yields warping functions defined on a finite open interval, since this warping function needs to satisfy certain differential inequalities that do not have global solutions. We have not been able to prove that one cannot construct warping functions defined on all of $\mathbb R$ with a different method, but suspect that this might be the case.

\begin{theorem}\label{thm:main5}
For any $0<m\leq M$, $k,K\in\RR$, $N > 1$, the class of covered Lorentzian pre-length spaces $(X,\mathcal{U})$ given below is precompact with respect to the Lorentzian Gromov--Hausdorff convergence as defined in \cite[Definition 3.12]{mondino-saemann2025}: 
\begin{itemize}
\item $X=\prescript{-}{}\RR \times_f Y$ is a globally hyperbolic, timelike non-branching, regular Lorentzian length space  which is a generalized cone with $m\leq f \leq M$, the global $\TCBA(k)$-condition holds, and $X$ can be equipped with a non-negative, full-support, Radon measure $\mfm$ such that it (and its causally reversed structure) satisfies the $\TMCP^e(K,N)$-condition.
\item $\mathcal{U}=\{U_n\}_{n\in\mathbb N}=\{I((t-n,y),(t+n,y))\}_{n\in \mathbb{N}}$ for some fixed $(t,y)\in X$. 
\end{itemize}
\end{theorem}

\begin{proof}
    Again we check points $(i)-(iii)$ from \cite[Theorem 6.2]{mondino-saemann2025}.

    $\underline{(i)}:$ This holds because the covers $\mathcal U$ were specifically constructed such that $\diam^\tau U_n=2n$.

    $\underline{(ii)}:$ This follows immediately from \Cref{thm:bound on eps net}.

    $\underline{(iii)}:$ This can be proven analogously to the proof of \Cref{thm:diamond precompact}.
    
\end{proof}

\subsection{Minkowski products}\label{ss:minkowski products}
In this subsection we exchange the global $\TCBA(k)$-condition for the assumption that $X$ is the Minkowski product of an interval $J$ with a geodesic space $Y$.

\begin{lemma}
Let $X = \prescript{-}{}J\times Y$ such that $Y$ is a geodesic space and let $\Lambda>0$, $\rho > 1$. Then all diamonds in $\hat{\mathcal{I}}(X)$ satisfy the $(\rho,\rho^{-1}/2,\Lambda)$-condition. 
\end{lemma}

\begin{proof}
Let $I_1\subset I_2$ be chronological diamonds in $\hat{\calI}(X)$ such that $I_2$ is an aligned $\rho$-enlargement of $I_1$ and $\diam^\tau(I_2) = \Lambda$. Therefore, without loss of generality, we may assume that 
\[
I_1 = I((0,y),(\rho^{-1}\Lambda,y)),\quad I_2 = I((0,y),(\Lambda,y))
\]
for some $y\in Y$. 

Let $\xi=(\alpha,\beta)\colon [0,B]\to X$ be a future-directed timelike geodesic from $\xi(0)=(0,y)$ to $\xi(B)$, where $\xi(B)\ll (\Lambda,y)$. 
By \Cref{thm:nice parametrizations}, we can assume that $\xi$ is parametrized by arclength, and since $f\equiv 1$, we also have that $\alpha(t) = \lambda t$ for some constant $\lambda>0$. Thus, by \Cref{lem:chronological relation generalized cones}, $\xi(B)\ll (\Lambda,y)$ implies
\begin{equation}\label{eq:chronological relation generalized cones 1}
d_Y(y,\beta(B)) < \Lambda - \alpha(B) = \Lambda - \lambda B.
\end{equation}

Moreover, for any $c\in (0,1)$ we have $\xi(cB)\gg (0,y)$ since $\xi$ is future-directed and timelike. Therefore,
\begin{equation}\label{eq:chronological relation generalized cones 3}
d_Y(y,\beta(cB)) < \alpha(cB) = \lambda cB.
\end{equation}

In particular, let $c = \frac{1}{2}\rho^{-1}$. Then, by \eqref{eq:chronological relation generalized cones 1}, we have
\[
d_Y(y,\beta(cB)) < \lambda cB < c\Lambda 
\]
which in turn implies
\[
d_Y(y,\beta(cB)) < 2c\Lambda - \lambda cB = \rho^{-1}\Lambda - \alpha(cB).
\]
In other words, $\xi(cB)\ll (\rho^{-1}\Lambda,y)$. 
\end{proof}

Note that in the previous theorems we only used the global $\TCBA(k)$-condition to apply \Cref{thm:generaldoubling}, which in turn only relies on curvature bounds to make sure the space satisfies the $(\rho,c,\Lambda)$-condition. Hence if this condition is satisfied for some $c$ regardless of curvature, the curvature assumption can be dropped. We summarize this in the following statement.

\begin{theorem}\label{thm:main6}
    The statements of \Cref{thm:bound on eps separated sets}, \Cref{thm:bound on eps net}, \Cref{thm:diamond precompact} and \Cref{thm:main5} hold true after exchanging the assumption that $X$ satisfies the $\TCBA (k)$-condition for the assumption that $X = \prescript{-}{}J\times Y$ is the Minkowski product of an interval $J$ with a geodesic space $Y$. 
\end{theorem}

\subsection{Measured Gromov--Hausdorff precompactness for generalized cones.}

As in the metric setting it is also possible to define a notion of measured Lorentzian Gromov--Hausdorff convergence, which does not only require convergence of the Lorentzian pre-length spaces, but also convergence of background measures in a certain sense. We will be using convergence in the measured Lorentzian Gromov--Hausdorff sense as defined in \cite[Definition 9.3, Definition 9.4]{mondino-saemann2025}. One can show that classes of spaces are precompact with respect to this notion of convergence similarly to how one shows this for the normal ``unmeasured'' notion. If $\{(X_i,\mathcal U_i,\mfm_i)\}_{i\in\Omega}$ is a class of covered measured Lorentzian pre-length spaces then by \cite[Theorem 9.6]{mondino-saemann2025} this is precompact with respect to the measured Lorentzian Gromov--Hausdorff convergence if points $(i)$ to $(iii)$ from \cite[Theorem 6.2]{mondino-saemann2025} are satisfied and in addition there exists a sequence of numbers $C_n>1$ such that 
\begin{equation}\label{eq:measured precompactness}
    \frac{1}{C_n}\leq\mfm(U_n)\leq C_n,
\end{equation}
for all $U_n\in\mathcal U_i$ and all $i\in\Omega$. This yields ``measured'' versions of \Cref{thm:diamond precompact} and \Cref{thm:main5}:

\begin{theorem}\label{thm:measured precompactness}
    For any $0<m\leq M<\infty$, $k,K\in\RR$, $N > 1$, $c,C>0$, $n_0\in\mathbb N$, the class of covered measured Lorentzian pre-length spaces $(X,\mathcal{U},\mfm)$ described below is precompact with respect to the measured Lorentzian Gromov--Hausdorff convergence:
    \begin{itemize}
        \item $X=I((\bar s, \bar y),(\bar t,\bar y))$ is a subset of the globally hyperbolic, timelike non-branching, regular Lorentzian length space $\prescript{-}{}J\times_f Y$, which is a generalized cone with $m\leq f \leq M$, $\diam^\tau(X)\leq\Lambda<D_k/2$, the global $\TCBA(k)$-condition holds, and $X$ (and its causally reversed structure) satisfies the $\TMCP^e(K,N)$-condition.
        \item $\mathcal{U}=\{U_n\}_{n\in\mathbb N}=\{I((\bar s+1/n,\bar y),(\bar t-1/n,\bar y))\}_{n\in \mathbb{N}}$.
        \item $c\leq\mfm(U_{n_0})\leq C$.
    \end{itemize}
\end{theorem}

\begin{proof}
    It is enough to check \eqref{eq:measured precompactness}. Without loss of generality let $n_0=1$. Pick $C_1:=\max\{1,C,1/c\}$ then clearly \eqref{eq:measured precompactness} holds for $n=1$. Moreover by the construction of the sets $U_n$ the ratio $\diam^\tau U_{n+1}/\diam^\tau U_n$ is bounded by some number $\rho_n$. 
    Hence we can apply \Cref{thm:generaldoubling} to obtain a constant $\tilde C_n$ such that 
    $$\mfm(U_n)\leq \tilde C_n\mfm(U_1)\leq\tilde C_n\,C_1=:C_n.$$
    Therefore the second inequality in \eqref{eq:measured precompactness} holds and so does the first since both $C_n$ and $\mfm(U_n)$ are increasing sequences. 
\end{proof}

\begin{theorem}
For any $0<m<M$, $k,K\in\RR$, $N > 1$, $c,C>0$, the class of covered measured Lorentzian length spaces $(X,\mathcal{U},\mfm)$ described below is precompact with respect to the measured Lorentzian Gromov--Hausdorff convergence as defined in \cite[Definition 3.12]{mondino-saemann2025}: 
\begin{itemize}
\item $X=\prescript{-}{}\RR \times_f Y$ is a globally hyperbolic, timelike non-branching, regular Lorentzian length space  which is a generalized cone with $m\leq f \leq M$, the global $\TCBA(k)$-condition holds, and $X$ (and its causally reversed structure) satisfies the $\TMCP^e(K,N)$-condition.
\item $\mathcal{U}=\{U_n\}_{n\in\mathbb N}=\{I((t-n,y),(t+n,y))\}_{n\in \mathbb{N}}$ for some fixed $(t,y)\in X$. 
\item $c\leq \mfm(U_1)\leq C$.
\end{itemize}
\end{theorem}

\begin{proof}
    This can be proven similarly to \Cref{thm:measured precompactness}.
\end{proof}

\section{Gromov's compactness using MCP in the fibre}
In this final section we obtain a compactness theorem for generalised cones where the metric fibre satisfies the $\MCP(K,N)$-condition for some $K, N$. 
\begin{theorem}\label{thm:main7}
Let $N>1$, $K\in\mathbb{R}$ and let $\{M_i\}_{i\in\mathbb{N}}$, $\{T_i\}_{i\in \mathbb{N}}$ and $\{\varepsilon_i\}_{i\in\mathbb{N}}$ be positive sequences such that $\lim_{i\to \infty}\varepsilon_i=0$. Let $\mathcal{X}$ be the class of covered Lorentzian pre-length spaces $(X,\mathcal{U})$ satisfying the following conditions:
\begin{enumerate}[label=(\roman*)]
\item $X$ is a generalized cone, i.e.\ $X = \prescript{-}{}J\times_f Y$.
\item $(Y,d,\mfm)$ is an essentially non-branching  metric measure space satisfying the $\MCP(K,N)$-condition.
\item $\mathcal{U} = \{U_i\}_{i\in\mathbb N}=\{(a_i,b_i)\times B_{i}(y_0)\}_{i\in\mathbb{N}}$ for some $y_0\in Y$ and $(a_i,b_i)\subset J$ with $|a_i-b_i| < T_i$, $(a_i,b_i)\subset (a_j,b_j)$ if $i\leq j$, and $(a_i-\varepsilon_i,b_i+\varepsilon_i) \subset J$ for all $i\in\mathbb N$.
\item $f\leq M_i$ on $(a_i-\varepsilon_i,b_i+\varepsilon_i)$.
\end{enumerate}
Then $\mathcal{X}$ is precompact with respect to the Lorentzian Gromov--Hausdorff convergence.
\end{theorem}

\begin{proof}
First we claim that, for any $(t,y)\in X\in \mathcal{X}$ and sufficiently small $\varepsilon>0$, any $\varepsilon$-causal diamond of the form $J((t-\varepsilon/2,y),(t+\varepsilon/2,y))$ contains a box of the form $(t-\varepsilon',t+\varepsilon')\times B_{\varepsilon''}(y)$ for any 
\[\varepsilon' < \varepsilon/2 \quad \text{and}\quad \varepsilon'' < \frac{1}{M}\left(\frac{\varepsilon}{2}-\varepsilon'\right)
\] 
where $M = \sup f([t-\varepsilon/2,t+\varepsilon/2])>0$. Indeed, if $(t',y')\in (t-\varepsilon',t+\varepsilon')\times B_{\varepsilon''}(y)$ then
\[
d(y',y) < \varepsilon'' < \frac{1}{M}\left(\frac{\varepsilon}{2}-\varepsilon'\right) \leq \min\left\{\int_{t-\varepsilon/2}^{t'}\frac{1}{f},\int^{t+\varepsilon/2}_{t'}\frac{1}{f}\right\}.
\]
By \Cref{lem:chronological relation generalized cones} this implies that $(t',y')\in J((t-\varepsilon/2,y),(t+\varepsilon/2,y))$ as claimed.

Now, we prove that $\mathcal{X}$ is \textit{uniformly totally bounded}, i.e., elements of $\mathcal{X}$ can be covered with $\varepsilon$-nets whose cardinality is bounded by a constant that only depends on $\varepsilon$. 
Let $i\in\mathbb{N}$ and $\varepsilon>0$ be such that $\varepsilon <\varepsilon_i$. Let $(t,y)\in U_i$ and observe that the $\varepsilon$-causal diamond $J((t-\varepsilon/2,y),(t+\varepsilon/2,y))$ is well defined since \[a_i-\varepsilon_i<t-\varepsilon/2<t+\varepsilon/2<b_i+\varepsilon_i.\]
Moreover, such diamond clearly intersects $U_i$. We will now show that we can cover $U_i$ with a number of such $\varepsilon$-causal diamonds that is uniformly bounded with respect to $i$ and $\varepsilon$.

Indeed, by the claim above, we know that
\begin{equation}\label{eq:box contained in diamond}
(t-\varepsilon',t+\varepsilon')\times B_{\varepsilon''}(y)\subset J((t-\varepsilon/2,y),(t+\varepsilon/2,y)),
\end{equation}
for $\varepsilon'< \frac{\varepsilon}{2}$ and $\varepsilon'' < \frac{1}{M_i}\left(\frac{\varepsilon}{2}-\varepsilon'\right)$. We can assume that $\varepsilon'$ and $\varepsilon''$ depend on $i$ and $\varepsilon$ by setting, for example, $\varepsilon' = \frac{\varepsilon}{4}$ and $\varepsilon'' = \frac{\varepsilon}{5M_i}$.

Since $|a_i-b_i| < T_i$ then $(a_i,b_i)$ can be covered with a uniformly bounded number, say $A(i,\varepsilon)>0$, of intervals of the form $(t-\varepsilon',t+\varepsilon')$ with $t\in (a_i,b_i)$ (actually, we can take $A(i,\varepsilon)= 1+\lceil T_i/\varepsilon' \rceil$). On the other hand, by the $\MCP(K,N)$-condition on $Y$, the Bishop--Gromov inequality implies that there exists a constant $B(i,\varepsilon)>0$ such that $B_i(y_0)$ can be covered with $B(i,\varepsilon)$ balls of the form $B_{\varepsilon''}(y)$ with $y\in B_i(y_0)$.

Thus, it is possible to cover $U_i$ with $C(i,\varepsilon) = A(i,\varepsilon)B(i,\varepsilon)>0$ boxes of the form $(t-\varepsilon',t+\varepsilon')\times B_{\varepsilon''}(y)$ with $(t,y)\in U_i$. Finally, each of these boxes is covered by an $\varepsilon$-causal diamond as in equation \eqref{eq:box contained in diamond}. Therefore, $U_i$ is covered by at most $C(i,\varepsilon)$ diamonds, and the result follows.
\end{proof}

\begin{remark}
    In the above theorem the assumption that the spaces $(Y,d,\mfm)$ satisfy the $\MCP(K,N)$-condition could be replaced by the the assumption that they satisfy the doubling property with uniform doubling constant $C>0$. Indeed the argument that there exists a uniform bound for how many balls of radius $\varepsilon$ you need to cover the ball $B_i(y_0)$ works analogously in this setting. 
\end{remark}

While the above result implies the compactness of certain classes of generalized cones with a curvature bound in the fiber, in Lorentzian geometry it is more natural to ask for a timelike curvature bound. However, these two kinds of curvature bounds are related and in specific cases imply each other. As a result we have the following Corollary.  

\begin{corollary}\label{cor:main8}
    Let $N>1$, $K,\kappa\in\mathbb R$ and $\{M_i\}_{i\in\mathbb{N}}$, $\{T_i\}_{i\in \mathbb{N}}$, $\{\varepsilon_i\}_{i\in\mathbb{N}}$ positive sequences such that $\lim_{i\to \infty}\varepsilon_i=0$. Let $\mathcal{X}$ be the class of covered Lorentzian pre-length spaces $(X,\mathcal{U})$ satisfying the following conditions:
    \begin{enumerate}[label=(\roman*)]
        \item $X=\prescript{-}{}J\times_fY$, where $f:J\to[0,\infty)$ is smooth, $\partial J=f^{-1}(\{0\})$, and $(Y,\mfm)$ is a proper, essentially non-branching, complete and geodesic metric measure space with a Radon measure $\mfm$.
        \item $(X,\mfm_X)$ satisfies $\TCD^e_p(\kappa ,N+1)$ and is timelike $p$-essentially non-branching, where $\mfm_X:=f(t)^Ndt\otimes d\mfm$.
        \item $\mathcal{U} = \{U_i\}_{i\in\mathbb N}=\{(a_i,b_i)\times B_{i}(y_0)\}_{i\in\mathbb{N}}$ for some $y_0\in Y$ and $(a_i,b_i)\subset J$ with $|a_i-b_i| < T_i$, $(a_i,b_i)\subset (a_j,b_j)$ if $i\leq j$, and $(a_i-\varepsilon_i,b_i+\varepsilon_i) \subset J$ for all $i\in\mathbb N$.
        \item $f\leq M_i$ on $(a_i-\varepsilon_i,b_i+\varepsilon_i)$ and $K\leq\sup_J\{-(f')^2+f''f\}$.
    \end{enumerate}
    Then $\mathcal{X}$ is precompact with respect to the Lorentzian Gromov--Hausdorff convergence. 
\end{corollary}

\begin{proof}
    For any $\prescript{-}{}J\times_f Y\in \mathcal X$ the fiber $Y$ satisfies the $\cd(K(N-1),N)$-condition by \cite[Theorem 6.1]{calisti-ketterer-saemann2025}. Hence we can apply \Cref{thm:main7} to obtain precompactness.
\end{proof}

\bibliographystyle{alphaurl}
\bibliography{references}
\end{document}